\documentclass{article}
\usepackage{cite}
\usepackage{amsmath,amssymb,amsfonts}
\usepackage{algorithmic}
\usepackage{graphicx}
\usepackage{textcomp}
\usepackage{xcolor}
\usepackage{color,soul}
\def\BibTeX{{\rm B\kern-.05em{\sc i\kern-.025em b}\kern-.08em
    T\kern-.1667em\lower.7ex\hbox{E}\kern-.125emX}}

\usepackage[utf8]{inputenc}
\usepackage{caption}
\usepackage{subcaption}
\usepackage{mathtools}
\usepackage{thmtools}
\usepackage{algorithm}
\usepackage{bm}
\usepackage{comment}
\usepackage{framed}
\usepackage{mdframed}
\usepackage[colorlinks=true,linkcolor=blue]{hyperref}

\newcommand{\RR}{\mathbb{R}}
\newcommand{\NN}{\mathbb{N}}

\newcommand{\KK}{\mathbb{K}}

\newcommand{\TT}{\mathbb{T}}
\newcommand{\EE}{{\mathbb{E}}}

\newcommand{\sH}{\mathcal{H}}
\newcommand{\sV}{\mathcal{V}}
\newcommand{\sR}{\mathcal{R}}

\newcommand{\calP}{{\mathcal{P}}}

\newcommand{\calM}{\mathcal{M}}
\newcommand{\calN}{\mathcal{N}}

\newcommand{\calL}{{\mathcal{L}}}
\newcommand{\calB}{{\mathcal{B}}}
\newcommand{\Ev}{\mathbb{E}}
\newcommand{\Res}{\Tr}
\newcommand{\Tr}{{T}}
\newcommand{\Ext}{\mathcal{E}}
\newcommand{\rH}{\mathcal{R}}

\newcommand{\knl}{\mathfrak{K}}
\newcommand{\rnl}{\mathfrak{R}}

\newtheorem{theorem}{Theorem}
\newtheorem{proof}{Proof}

\begin{document}

\title{Kernel Methods for Regression in Continuous Time over Subsets and Manifolds
}

\author{Nathan Powell\thanks{ Department of Mechanical Engineering, Virginia Tech, Blacksburg. VA email: {nrpowell@vt.edu, boweiliu@vt.edu, kurdila@vt.edu}}\and  Jia Guo\thanks{ Department of Mechanical Engineering, Geogia Tech, Atlanta, GA. email: {jguo@gatech.edu}} \and Sai Tej Parachuri \thanks{Department of Mechanical Engineering and Mechanics, Lehigh University, Bethlehem, PA.  
email: {saitejp@lehigh.edu}} \and John Burns \thanks{Interdisciplinary Center for Applied Mathematics (ICAM), Virginia Tech, Blacksburg, VA 24060, USA}  \and  Boone Estes \footnotemark[1] \and Andrew Kurdila \footnotemark[1] }
\maketitle

\begin{abstract}
This paper derives error bounds for regression in continuous time over subsets of certain types of Riemannian  manifolds.The regression problem is typically driven by a nonlinear evolution law taking values on the manifold, and it is cast as one of optimal estimation in a reproducing kernel Hilbert space (RKHS). A new notion of persistency of excitation (PE) is defined for the estimation problem over the manifold, and rates of convergence of the continuous time estimates are derived using the PE condition.  We discuss and analyze two approximation methods of the exact regression solution. We then conclude the paper with some numerical simulations that illustrate the qualitative character of  the computed function estimates. Numerical results from function estimates generated over a trajectory of the Lorenz system are presented. Additionally, we analyze an implementation of the two approximation methods using motion capture data.
\end{abstract}


\section{Introduction}
\subsection{Motivation}
The study of machine or statistical learning theory, and its application to regression problems, has been a topic of interest for years. \cite{gyorfi}. These  techniques have had a lasting impact in Bayesian estimation and estimation using Gaussian processes. Many of  these efforts in machine or statistical learning theory, Bayesian estimation, and Gaussian processes have theoretical foundations that exploit formulations cast in terms  of reproducing kernel Hilbert spaces (RKHS), which are also known as native spaces.  \cite{williams} While some recent efforts including \cite{qw2021,foster2020,theodorou2019} have sought to further understand learning theory in the context of dynamical systems theory, and vice-versa, it is accurate to say that most of the above work to date has focused on cases where the samples used for learning or regression are generated from some independent and identically distributed (IID), stochastic,  discrete measurement process.   A good account on the state-of-the-art in distribution-free learning theory and its focus on discrete processes can  be found in \cite{gyorfi,williams,scholkopf}. 

This paper seeks to use RKHS formulations in continuous time estimation  problems, in the spirit of \cite{sastryBook,krsticBook,ioannouBook,narendraBook,farrellBook,hovakimyanBook},  to achieve some of the advantages  that are so clear in the above applications of  learning theory to  IID discrete systems. The theory and algorithms in the references \cite{sastryBook,krsticBook,ioannouBook,narendraBook,farrellBook,hovakimyanBook} describe many of the working tools used by specialists in the field of adaptive estimation  and control theory as it is applied to ordinary differential equations (ODEs). As described in \cite{sastryBook,krsticBook,ioannouBook,narendraBook,farrellBook,hovakimyanBook}, it is standard that the regression problem in finite dimensional spaces is often used to motivate, explain, and study  adaptive estimation and control theory for ODEs. The regression problem arises then when the  ODE is characterized by a finite linear combination of known regressor functions.  In recent papers, the authors have introduced adaptive estimation problems in RKHS formulations in \cite{paruchuri2022sufficient,kgp2019limit,gpk2020pe,guo2020approx,paruchuri2022kernel,kepler}, where the evolution is described by a distributed parameter system (DPS) over a native space.  Here we study the related regression problem in continuous time in a native space, which plays an analogous role in the RKHS/DPS framework  to that when a finite dimensional collection of regressors appear in an ODE. 

One way to view this paper is as an exploration of what features or properties of the well-studied regression problem that underlies adaptive estimation and control of ODEs also hold, or can be extended to,  the regression problem  over a native space. This paper is also an attempt to address some of the open questions summarized in \cite{qw2021}, for instance, that relate learning theory and  dynamical systems theory. 

\subsection{Problem Description}
In this paper we study a regression problem in continuous time where an approximating agent traverses the configuration space $X$ along a trajectory $t\mapsto \phi(t)\in X$ making observations $y(t) = G(x(t))\in \RR$ of some unknown function $G:X\rightarrow\RR$. The configuration space $X$ is always a complete metric space, but it need not be compact.  The most important cases discussed in the paper choose $X:= \RR^d$ or some other smooth Riemannian manifold.  The system that generates the trajectory can be quite general. Any generally nonlinear autonomous or nonautonomous system may generate the trajectory $t \mapsto \phi(t) \in X$. We only require that the trajectory is continuous.

In this paper, we assume that the input/output history 
$\{x(\tau),y(\tau)\}_{\tau \in [0,t]}$ is observed without noise.   We concentrate on this paper on how the choice of a function space and geometric properties of the flow influence the rate of convergence of approximations of the solution of the regression problem in continuous time, which is challenging enough for a single paper.    We address the effects of uncertainty in the continuous time regression problem using the theory of inverse problems in a forthcoming paper.

While the flow is defined on $X$, which may not be compact, approximations of the regression problem will be carried out over some typically compact subset $S \subseteq X$. Two interpretations of the set $S \subseteq X$ are possible. It may be that the compact set $S$ is some known,  prescribed subdomain over which approximations are sought. In the most difficult problem setting, however, the set $S$ is not known {\em a priori} but represents an emergent structure. Over time, samples along the trajectory  accumulate in $S \subseteq X$. As time progresses, we obtain more and more information about the structure of $S$, but initially we may not have any idea about its structure. In the problem at hand, it can be the case that $S$ is highly irregular. 

Two examples are typical of the abstract situation above. In the first, $X := \RR^d$ and $S \subseteq X$ is some compact subset. In the numerical examples in Section \ref{sec:examples}, the Lorenz system is of this type. We have $X =\RR^3$ and $S \subset \RR^3$ is an irregular, unknown, compact, positively invariant set. 

Another important example arises when $S \subseteq \mathcal{M}$ and $\mathcal{M}$ is a compact, smooth, Riemannian manifold that is regularly embedded in $X$. Of course, if $X$ is compact, it is always possible to choose $S\subseteq \calM\equiv X$. We emphasize that we  reserve the notation $\calM$ for a compact manifold in this paper.  Again, in the most difficult form of this problem, the manifold $\mathcal{M}$ may not be known {\em a priori}. That is, we may not have explicit  knowledge of the specific coordinate charts that define $\mathcal{M}$, but rather only that it  is embedded in the larger manifold  $X$, $\mathcal{M} \subseteq X$. 

One important underlying goal should be clear in view of the comments above describing $S,\calM,X$. In a sense, we seek estimation methods that are robust  with respect to  uncertainty in the knowledge about the underlying unknown subset or manifold supporting the dynamics. The assumption where the form of $\mathcal{M}$ is unknown has recently been studied by the authors in \cite{powell2022koopman} when seeking data-dependent approximations of the Koopman operator. 

Other  progress on a related problem using online, gradient learning laws has been reported by the authors in \cite{pgk2020suff,kgp2019limit,gpk2020pe,gpk2020approx}. In these papers the method of native space or RKHS embedding is used to generate online estimates in applications to adaptive estimation and control theory.  Here in contrast we study an offline, optimal estimation approach. In comparison to the now familiar approaches for parametric estimation in finite dimensional Euclidean spaces, this paper derives convergence results for regression estimates in a reproducing kernel Hilbert space $\sH$ that is generated by a known, admissible kernel $\knl:X\times X \rightarrow \RR$.  The native space $\sH$ can be interpreted as the Hilbert space that contains all functions that can be represented as the limit of the translates of a certain template function. Given the kernel $\knl$ that defines $\sH$, we define the kernel section or basis function centered at $x$ as $\knl_x(\cdot) := \knl(x,\cdot)$. Then, $\sH$ is defined to be $\sH = \overline{\text{span}\{\knl_x \ | \ x \in X\}}$ the closed linear space as the kernel basis moves around in $X$. This paper can be viewed as the extension of standard results in Euclidean spaces as in \cite{sastryBook,krsticBook,ioannouBook}, see page 48 of  \cite{sastryBook} for Chapter 4 of \cite{ioannouBook} for instance,  to the case when an agent generates estimates in continuous time of a function in the native space $\sH$ defined over a subset $S$ manifold of a manifold $\calM$ or $X$. 

A primary contribution of this paper is the characterization of the error in continuous time using methods from scattered data approximation in kernel spaces.     Another contribution is the introduction of a new PE condition that is well-defined over manifolds and that enables the analysis of convergence of the time-varying regression estimate. We review these contributions in some detail next.

\subsection{Summary of New Results}
\label{sec:resultssummary}

There are three specific new results derived in this paper that are not addressed in any of the previous papers by the authors in \cite{pgk2020suff,kgp2019limit,gpk2020pe,gpk2020approx,powell2022koopman}, or in the literature at large.  Suppose that $\phi: t \mapsto \phi(t) \in X$ is a trajectory of either an autonomous or nonautonomous flow on the manifold. The regression problem described above is solved using the operator $T_\phi(s,t):\sH \rightarrow  \sH$
\[
T_\phi(s,t):=\int_s^t \knl_{\phi(\tau)}\otimes \knl_{\phi(\tau)} \nu(d\tau).
\]
where $\knl_{\phi(\tau)}:=\knl(\phi(\tau),\cdot)$ is the kernel basis function centered at $\phi(\tau)$ and $\nu$ is a finite measure on $[s,t]$. The tensor product operator $\knl_{\phi(\tau)}\otimes \knl_{\phi(\tau)} $ satisfies $\knl_{\phi(\tau)}\otimes \knl_{\phi(\tau)} g = \knl_{\phi(\tau)}\langle \knl_{\phi(\tau)}g\rangle_\sH$ for all $g \in \sH$. The first new result is summarized in Theorem \ref{th:Tphi} where sufficient conditions are given that ensure that this operator is compact, positive, and self-adjoint. This generalizes a result in the Appendix in \cite{devito2014} to the time-dependent case, which is essential to the study of the regression problem in continuous time.
The second new result is the introduction of a new persistency condition in 
Equation \ref{eq:PE} for flows over a manifold that generalizes the one in our earlier papers. It defines persistency for a general closed subspace $\sV\subseteq \sH$, where the norm on $\sV$ that can be different than the norm on $\sH$. The  publications \cite{begklpp2022,pgk2020center} always make the special choice $\sV:=\sH_S$ where $\sH_S$ is the native space generated by a subset $S\subset X$. The generalization in this paper is essential to prove convergence of estimates in certain spectral approximation spaces $A^r \subseteq \sH$, which depend on a trajectory $t\mapsto \phi(t)\in X$.

Finally, when the new PE condition holds for the subspace $\sV\subseteq \sH$, we show that 
\[
\|\hat{g}_\sV(t,\cdot)-\Pi_\sV G\|_{\sH} \leq \frac{\bar{\knl}^2 m \Delta}{\gamma_1 m + \gamma}\|(I-\Pi_{\sV} ) G\|_\sH+\frac{\gamma}{\gamma_1 m + \gamma}\|\Pi_\sV G\|_\sH
\]
where $t\mapsto \hat{g}_\sV(t,\cdot) \in \sV$ is the optimal solution of the (offline) regression problem, and $\Pi_\sV$ is the $\sH$-orthogonal projection of $\sH$ onto $\sV$. The constant $\bar{\knl}$ is a bound on the reproducing kernel $\knl$ that defines $\sH$, the constants $\gamma_1$ and $\Delta$ arise in the PE condition in Equation \ref{eq:PE}, $\gamma$ is the regularization parameter in the continuous regression error functional, and the time $t:=m\Delta$ for the positive integer $m > 0$. Note that as time $t = m\Delta \to \infty$, the estimate above implies that $$\| \hat{g}_{\sV}(t,\cdot)-\Pi_{\sV} G\|_{\sH} \lesssim  \mathcal{O}\biggl(\|(I-\Pi_{\sH}) G\|_{\sH}\biggr).$$ Intuitively, the solution of the regression problem in continuous time under the new PE condition implies that it asymptotically approaches the projection over the PE subspace.  We further refine this estimate in some cases to show that, when $N$ samples are used to define certain finite dimensional spaces of approximants $\sV:=\sH_N$ and these spaces are PE,  we have 
\begin{align*}
 \|&\hat{g}_N(t,\cdot)-\Pi_N G\|_{\sH_S}  \\ &\leq \left(\left (\frac{\bar{\knl}^2 m \Delta}{\gamma_1 m + \gamma} \right )
\|\calP_N\|_{L^2(S)}\right)\|G\|_{\sH_S} +  \frac{\gamma}{\gamma_1 m + \gamma}\|\Pi_N G\|_{\sH_S}. 
\end{align*}
for all unknown functions $G$ that are smooth enough. This bound makes use of the power function $\calP_N(x)$, over the set $S$, that is defined as 
\[
\calP_N(x):=|\knl(x,x)-\knl_N(x,x)| \quad \text{ for all } x\in S.
\]
In this expression $\knl_N$ is the kernel that defines the native space $\sH_N$ of approximants, and $S$ is the closure of the trajectory $\tau\mapsto \phi(\tau)$ in $X$. The kernel $\knl_N(x,y) := (\Pi_N k_x)(y)$ by definition \cite{berlinet}.   It is worth noting that this  error bound for the regression problem in continuous time has some similarity to that in \cite{gao2019gaussian, bai2016bayesian}. These papers derive pointwise error bounds for discrete regression or Bayesian estimation for discrete time processes, in contrast to the integrated error bound above for systems in continuous time. The relationship of the solution of the continuous time regression problem to the more familiar discrete IID, stochastic, case is discussed in detail in Section \ref{sec:learningTheory}.

\subsection{Notation, Symbols, and Background}
\label{sec:notation}
In this paper the state space $X$  is a complete metric space. The most important examples choose X to be the Euclidean space $X:=\RR^d$, a smooth and compact Riemannian manifold $X:=\mathcal{M}$, or certain measurable subsets of these. 
We denote by $\knl:X\times X\rightarrow \RR$ a   symmetric, nonnegative, continuous kernel that induces the scalar-valued native space $\sH$ of functions defined over $X$.
Throughout the paper  $\sH$ is a reproducing kernel Hilbert space (RKHS) of real-valued functions over the set $X$ that is given by $\sH:=\overline{\text{span}\{\knl(x,\cdot)\ | \ x\in X\}}$.

In this paper, we often must refer to time-varying functions that take values in $\sH$. We write $f(t,\cdot)$ to represent the spatial function $x \mapsto f(t,x)$ for fixed time $t$. That is $f(t,\cdot)$ for each fixed time $t$.

We write $\EE_X:\sH \to \RR$ for the evaluation functional at $x \in X$, which satisfies $\EE_x f := f(x)$ for each $f \in \sH$. The adjoint $\EE_x^*:\RR \to \sH$ can be understood as the multiplication operator given by $\EE_x^* \alpha := \knl(x,\cdot)\alpha $ for all $\alpha \in \RR$. {We denote by $\calL(\sH)$ the linear and bounded operators that map from $\sH$ to $\sH$. The notation $\calB_{\calL(\sH)}$ denotes the Borel $\sigma$-algebra on $\calL(\sH)$.} 

For any subset $S\subseteq X$ we define the native space $\sH_S\subseteq \sH$ generated by  $S$ as $\sH_S:=\overline{\text{span}\{\knl(x,\cdot)\ | \ x\in S\}}$. 
We emphasize that $\sH_S$ is {not} the Hilbert space that consists of  restrictions of functions to $S$: since $\sH_S\subseteq \sH$ functions in $\sH_S$ are supported on $X$. 
The space $\sH_S$ is a native space having kernel $\knl_S(x,y):=\langle\Pi_S \knl(x,\cdot),\Pi_S(\knl(y,\cdot)\rangle_\sH$ for all $x,y\in X$,  with $\Pi_S$ the $\sH$-orthogonal projection of $\sH$ onto $\sH_S$.  

We denote by $\Res_S$ the trace or restriction operator $\Res_S g:=g|_{S}$. The space of restrictions $\sR_S=\Res_S(\sH)$ is an RKHS with the restricted kernel $\rnl(x,y):=\knl(x,y)$ for all $x,y\in S\subseteq X$.  There is a canonical minimum norm extension operator $\Ext_S:\rH_S\rightarrow \sH$ that satisfies $\Ext_S \Res_S =\Pi_S$.  This operator is an isometry $\Ext_S:\sR_S \rightarrow \sH_S$ and satisfies 
$$
\|\Ext_S g\|_{\sH}= \|g\|_{\sR_S}=\inf \left \{ \|f\|_{\sH} \ | \ g=\Res_S f, f\in \sH  \right \}
$$
for all $g\in \sR_S$.

\section{Regression in Continuous Time}
\label{sec:regresscontinuous}
In this section we study in detail the problem of regression in continuous time in a RKHS $\sH$ of real-valued functions over the state space $X$, where $X$ is a complete metric space. The trajectory $t \mapsto \phi(t)$ of the system is assumed to be continuous. The overall situation is depicted graphically in Figure \ref{fig:continuousregression}.  In this problem we are given a trajectory $t\mapsto \phi(t)\in X$, and we make measurements $y(\tau)=G(\phi(\tau))$ of an unknown function $G\in \sH$ at $\phi(\tau)\in X$ for each $\tau\in [0,t]$. 
\begin{figure}[ht!]
    \centering
    \includegraphics[width=0.8\textwidth]{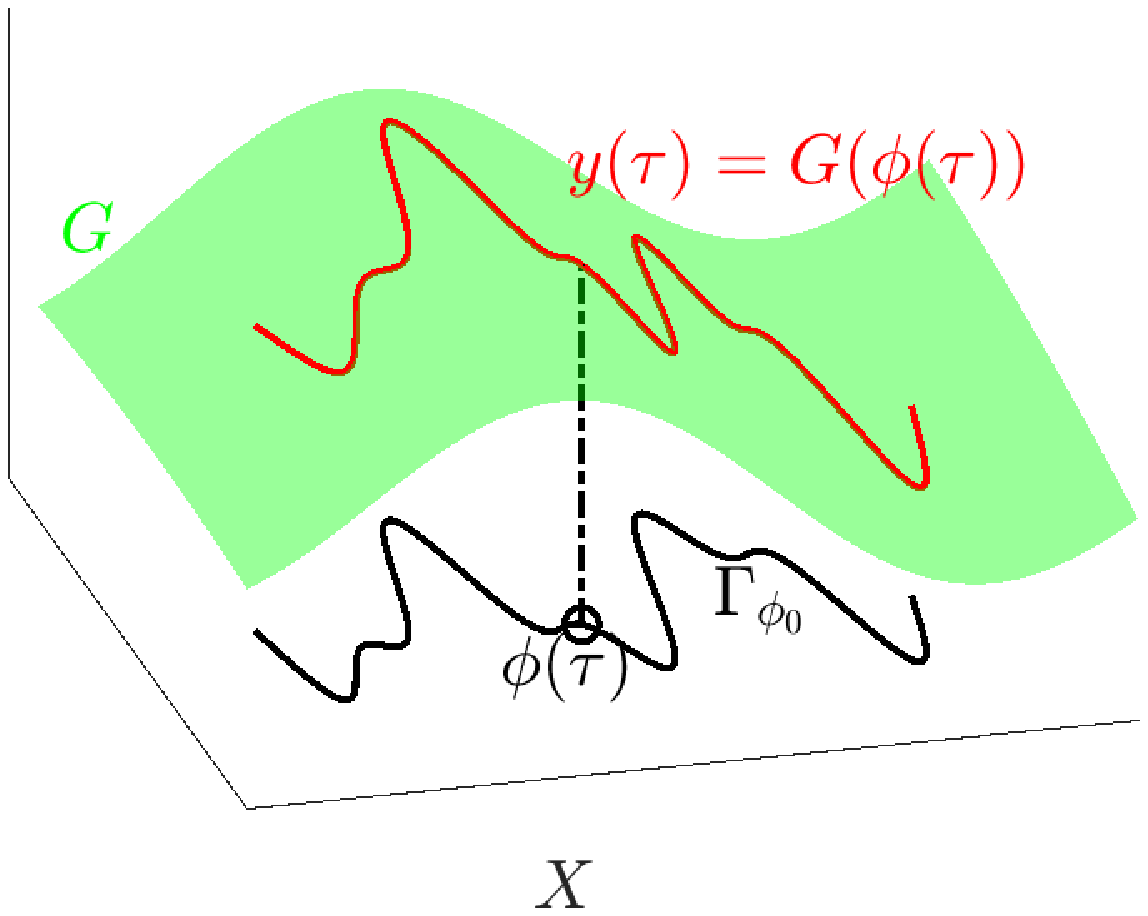}
    \includegraphics[width=0.8\textwidth]{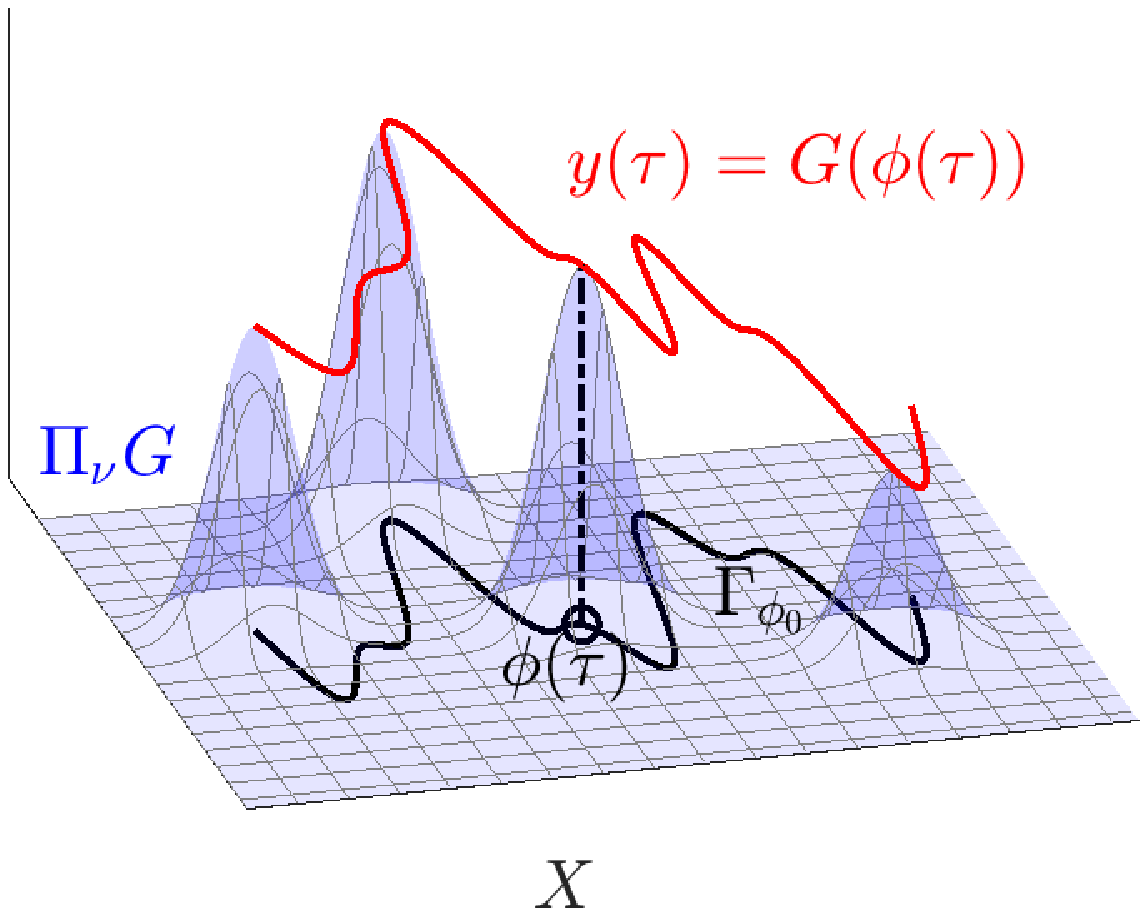}
    \caption{(Top) An illustration of regression in continuous time. The orbit $\Gamma(\phi_0) = \bigcup_{\tau \geq 0} \phi(\tau)$ is a subset of the state space $X$. The output is determined by a function $G$ represented by the green surface over $X$. (Bottom) An illustration of the estimate $\Pi_\nu G$ represented by the blue mesh is generated by a subspace $\sV$ of $\sH$. Here the subspace consists of kernel functions illustrated by the ``bumps" with centers at various points along the orbit. In this figure, one center is represented by $\phi(\tau)$.}
    \label{fig:continuousregression}
\end{figure}
The goal is to use the continuous collection of samples $\{(\phi(\tau),y(\tau))\}_{\tau\in [0,t]}$ to build a time-varying estimate $\hat{g}(t,\cdot)\in \sH$ of the unknown function $G$. 

\subsection{The Offline Optimal Regression Estimate in an RKHS} 
\label{sec:offlineoptimal}
 We define the integral error functional $E:\RR^+\times \sH \times C([0,t],X)\rightarrow \RR$ to be 
\begin{equation*}
    E(t,g;\phi):=\frac{1}{2} \int_0^t |y(\tau)-\Ev_{\phi(\tau)}g|^2 \nu(d\tau) + \frac{1}{2}\gamma \|g\|_{\sH}^2
\end{equation*}
for the state space $X$, some measure $\nu$ on $\RR^+$, and a regularization parameter $\gamma>0$. In this equation $\EE_{\phi(\tau)}:\sH \to \RR$ is the evaluation functional at $\phi(\tau)$, which satisfies $\EE_{\phi(\tau)}g = g(\phi(\tau)) $ for any $g \in \sH$. The error functional $E(t,g;\phi)$ can be rewritten as 
\begin{align*}
    E(&t,g;\phi)=\frac{1}{2} \int_0^t |\Ev_{\phi(\tau)}(G-g)|^2 \nu(d\tau) + \frac{1}{2} \gamma  \|g\|_{\sH}^2, \\
    &= \frac{1}{2}\int_0^t\left \langle  \Ev_{\phi(\tau)}^*\Ev_{\phi(\tau)}(G-g),G-g\right\rangle_{\sH}\nu(d\tau) + \frac{1}{2}\gamma \|g\|_{\sH}^2 .
\end{align*}
The adjoint $\EE^*_{\phi(\tau)}:\RR \to \sH$ is given by $\EE^*_{\phi(\tau)} \alpha := \knl_{\phi(\tau)}\alpha$ for any $\alpha \in \RR$.
The study of the error functional $E(t,g;\phi)$ takes a familiar structure when we introduce an  operator $T_{\phi}{(s,t)}:\sH\rightarrow \sH$ via the identity
\[
T_{\phi}{(s,t)} g:=\int_s^t \Ev_{\phi(\tau)}^*\Ev_{\phi(\tau)}g\nu(d\tau). 
\]
Note that $T_{\phi}{(s,t)}$ is a time-varying operator that depends on the trajectory $t\mapsto \phi(t)$. In the following arguments, and later at several places in the text, the properties of the operator $T_\phi(s,t)$ are important. We summarize some of its properties  in the following theorem. 
\begin{theorem}
\label{th:Tphi}
Suppose that $\knl:X\times X\rightarrow \RR$ is a continuous admissible kernel that induces the native space $\sH$ of continuous real-valued functions over $X$. The operator $T_\phi(s,t)$ above is an integral operator 
\begin{align*}
(T_{\phi}{(s,t)} g)(\xi)&=\int_s^t \knl(\xi,\phi(\tau))\langle\knl_{\phi(\tau)},g\rangle_\sH \nu(d\tau)\\
&=\int_s^t \knl(\xi,\phi(\tau))g(\phi(\tau)) \nu(d\tau).
\end{align*}
If there is a constant $\bar{\knl}>0$ such that $\sqrt{\knl(x,x)}\leq \bar{\knl}$ for all $x\in X$   and the trajectory $\phi\in C([0,t],X)$ is continuous in time, then the operator-valued map 
\[
\tau \mapsto \knl_{\phi(\tau)}\otimes \knl_{\phi(\tau)} \in \calL(\sH)
\]
is continuous from $\TT:=[s,t]$ to $\calL(\sH)$ and measurable as a map from $(\TT,\calB_\TT)$ into $(\sH,\calB_{\calL(\sH)})$. The operator $T_\phi(s,t)$ can be understood as the Bochner integral
\begin{align}
    T_{\phi}(s,t)=\int_\TT  \knl_{\phi(\tau)}\otimes \knl_{\phi(\tau)} \nu(d\tau) = \int_\TT \Ev_{\phi(\tau)}^*\Ev_{\phi(\tau)}\nu(d\tau). \label{eq:bochnerT}
\end{align}
The operator $T_\phi(s,t)$ is compact, self-adjoint, positive, and trace class. 
\end{theorem}
\begin{proof}
Since $\knl$ is continuous and the trajectory $\tau\mapsto \phi(\tau)$ is continuous, the map $\tau\mapsto \knl_{\phi(\tau)}$ is continuous. This fact follows from the identity 
\begin{align*}
    \|\knl_{\phi(\tau)}-\knl_{\phi(t)}\|^2_{\sH} :=\knl(\phi(\tau),&\phi(\tau))-2\knl(\phi(\tau),\phi(t))\\
    &+\knl(\phi(t),\phi(t)),
\end{align*}
and the righthand side goes to zero as $\tau\rightarrow t$ by the continuity of $\tau\mapsto \knl(\phi(\tau),\phi(\tau))$.
This identity can then be used to show that the curve $\tau\mapsto \knl_{\phi(\tau)}\otimes \knl_{\phi(\tau)}$ is continuous as a map from $\TT \rightarrow \calL(\sH)$ from the expression
\begin{align*}
    &\left \|\left (\knl_{\phi(\tau)}\otimes \knl_{\phi(\tau)} - \knl_{\phi(t)}\otimes \knl_{\phi(t)}
   \right )h \right \| \\
   &\leq 
   \left \| \knl_{\phi(\tau)}\otimes \knl_{\phi(\tau)}h- \knl_{\phi(\tau)}\otimes \knl_{\phi(t)}h\right \| \\
   &\hspace*{.5in}+ \left \| \knl_{\phi(\tau)}\otimes \knl_{\phi(t)}h- \knl_{\phi(t)}\otimes \knl_{\phi(t)}h\right \|, \\
   & \leq 2\bar{\knl} \|h\| \| \knl_{\phi(\tau)}-\knl_{\phi(t)}\|.
\end{align*}
This collection of inequalities above makes repeated us of the bound $\|\knl_x\|\leq \bar{\knl}$. 
The measurability of the map from $(\TT,\calB_\TT)$ to $(\sH,\calB_{\calL(\sH)})$ then follows from the continuity of this map. 
Finally, the fact that the Bochner integral in Equation \ref{eq:bochnerT} exists in $\calL(\sH)$ is a consequence of the bound
\begin{align*}
    \int_\TT \|\knl_{\phi(\tau)}\otimes \knl_{\phi(\tau)}\|_{\calL(\sH)} \nu(d\tau) \leq \bar{\knl}^2 \nu(\TT).
\end{align*}

Next, we consider the compactness of $T_\phi(s,t)$. This proof essentially follows the same line of reasoning as that in Proposition 14 of \cite{devito2014}, which is carried out for the operator $\knl_x\otimes \knl_x$ and a {\em spatial measure $\mu$} on $X$. Since the operator $\knl_{\phi(\tau)}\otimes \knl_{\phi(\tau)}$ is finite rank for each $\tau$, it is trace class for each $\tau$, with
\[
\text{Tr}\left ( \knl_{\phi(\tau)}\otimes \knl_{\phi(\tau)} \right ) =\knl(\phi(\tau),\phi(\tau))\leq \bar{\knl}^2. 
\]
The trace operator is a continuous linear operator on the trace norm class, and we know that 
\begin{align*}
    \text{Tr} \left (\int_\TT \knl_{\phi(\tau)}\otimes \knl_{\phi(\tau)}  \nu(d\tau)\right )&= \int_\TT \text{Tr} \left ( \knl_{\phi(\tau)}\otimes \knl_{\phi(\tau)}\right ) \nu(d\tau)\\
    &\leq {\bar{\knl}^2} \nu(\TT)
\end{align*}
by the mapping property of a Bochner integral under continuous linear operators. The Bochner integral exists and is therefore trace class. 

We conclude this proof by showing that $T_\phi(s,t)$ is a positive operator. This is a modification of the analysis in \cite{cucker}, which treats a different problem where again the integral is over space, not time.  For completeness, we give its outline. Let $\TT_{N,k}$ be a family of measurable subsets of $\TT$ with $\TT=\cup_{k=1}^N\TT_{N,k}$ and $\nu(\TT_{N,k})= \nu(\TT)/N$.  Fix a quadrature point $t_{N,k}\in \TT_{N,k}$ from each set $\TT_{N,k}$ and define $\xi_{N,k}=\phi(t_{N,k})$. We then have 
{\small
\begin{align}
    &\left (T_\phi(s,t)h,h\right )= \notag \\
    &\lim_{N\rightarrow \infty} \sum_{i,j=1}^N \int_{\TT}\int_\TT \knl(\xi_{N,j},\xi_{N,i})h(\xi_{N,i})h(\xi_{N,j})\chi_{{N,i}}(\tau) \chi_{{N,j}}(\tau) \nu(d\tau) \notag  \\
    &=\lim_{N\rightarrow \infty} 
    \frac{1}{N^2} \sum_{i,j=1}^N\knl(\xi_{N,j},\xi_{N,i})h(\xi_{N,i})h(\xi_{N,j})\geq 0, \label{eq:lowerbound}
\end{align}
}

\noindent where $\chi_{N,i}$ is characteristic function of the subset $\TT_{N,i}$. 
The positivity  of $T_\phi(s,t)$ follows from the semidefiniteness of the kernel $\knl$. 
\end{proof}
With these properties in hand, we can derive the optimal offline regression estimate in continuous time. In the usual way we can ``complete the square'' and write
\begin{align*}
E(t,g;\phi)&=\frac{1}{2}\left\langle (T_{\phi}(0,t)+\gamma I)g,g\right\rangle_{\sH} - \left \langle T_{\phi}(0,t)G,g\right \rangle_{\sH} \\
& + \frac{1}{2}\left \langle T_{\phi}(0,t)G,G)\right\rangle_{\sH}.
\end{align*}
Now we define the best approximation, i.e. the regressor solution in continuous time, 
\[
\hat{g}(t,\phi):=\underset{g\in \sH}{\text{argmin}}\ E(t,g;\phi).
\]
But it is relatively easy to calculate the Gateaux derivative of this functional in the direction $h\in \sH$. By definition  it  is given by
\begin{align*}
\left \langle 
DE(t,g;\phi),h 
\right \rangle_{\sH}
&:=\lim_{\epsilon \rightarrow 0} \frac{E(t,g+\epsilon h;\phi)-E(t,g;\phi)}{\epsilon} \\
&=\left \langle
(T_\phi(0,t)+\gamma I)g-T_\phi(0,t)G,h 
\right \rangle_{\sH}.
\end{align*}
Local minima to the above minimization problem must satisfy 
\[
\hat{g}(t,\cdot)=\left ( 
T_\phi(0,t)+\gamma I
\right )^{-1}T_\phi(0,t)G,
\]
because the operator $T_\phi(0,t)+\gamma I$ is invertible as a map from $\sH\rightarrow \sH$.

\subsection{Galerkin Approximations of the Regression Estimate}
\label{sec:galerkinapprox}
The regression estimate $\hat{g}(t,\cdot)$ is the solution of an operator equation in the generally infinite dimensional space $\sH$. Practical algorithms must consider approximations of this solution. In this section we discuss one method for obtaining approximations based on Galerkin's method. 
A review of some of the common properties of Galerkin approximations in Hilbert spaces is given in the Appendix in Section \ref{sec:galerkinbackground}. The regression solution $\hat{g}(t,\cdot)$ satisfies the equation 
\begin{align}
    \left \langle \left ( T_\phi(0,t)+\gamma I \right )\hat{g}(t,\cdot),h\right \rangle_{\sH} &= \left \langle T_\phi(0,t)G,h \right \rangle_{\sH} \quad \text{ for all } h\in \sH.
    \label{eq:weakeq}
\end{align}
Let $\sH_N\subseteq \sH$ be some $N$-dimensional subspace of $\sH$ that is used to build approximations. The  Galerkin approximation $\hat{g}_N(t,\cdot)\in \sH_N$ is the solution of the analogous equation 
\begin{align}
    \left \langle \left ( T_\phi(0,t)+\gamma I \right )\hat{g}_N(t,\cdot),h_N\right \rangle_{\sH} &= \left \langle T_\phi(0,t)G,h_N \right \rangle_{\sH} \quad \text{ for all } h_N\in \sH_N. \label{eq:weakgalerkin}
\end{align}
For each $t\in \RR^+$ we define the bilinear form $a(t)(\cdot,\cdot):\sH\times \sH\rightarrow \RR$ to be $a(t)(g,h):=\left \langle \left ( T_\phi(0,t)+\gamma I \right )g,h\right \rangle_{\sH}$ for each $g,h\in \sH$. 
The bilinear form $a(t)(\cdot,\cdot)$ is bounded and coercive as defined in Section  \ref{sec:galerkinbackground} since 
\begin{align*}
    a(t)(g,h)&\leq (\bar{k}^2\nu([0,t]+\gamma)\|g\| \|h\|  & &\text{ for all } g,h \in \sH, \text{ and } \\
    a(t)(g,g)&\geq \gamma \|g\|^2  & & \text{ for all } g\in \sH.
\end{align*}
By the Lax-Milgram Theorem \ref{th:laxmilgram}
 there is a unique solution $\hat{g}(t,\cdot)$ of Equation \ref{eq:weakeq} and $\hat{g}_N(t,\cdot)$ of Equation \ref{eq:weakgalerkin}. From Theorem \ref{th:galerkinerror} we know that 
 \begin{align*}
     \|\hat{g}(t,\cdot) - \hat{g}_N(t,\cdot)\|\leq \frac{(\bar{k}^2\nu([0,t]) + \gamma)}{\gamma}\|(I-\Pi_N)G\|.
 \end{align*}

In the analysis so far, we have $\hat{g}(t,\cdot)\in \sH$, $\hat{g}_N(t,\cdot)\in \sH_N\subseteq \sH$, under the assumption that $G\in \sH$, with $\sH$ a native space of functions supported on the configuration space $X$. The  error in approximating the optimal regression estimate $\hat{g}(t,\cdot)$ by the Galerkin estimate $\hat{g}_N(t,\cdot)$ is bounded by the norm on the best approximation of $G$ from the subspace $\sH_N$, but there are a number of standard techniques  to build sharp bounds on the error $\|(I-\Pi_N)G\|$, which depend on how regular the function $G$ is. These methods can be based on spectral analysis of integral operators and Mercer kernels, properties of the power function, or versions of the many zeros theorems \cite{wendland,hangelbroek2012polyharmonic,fuselier2012}. We discuss such specific cases in the examples in Section \ref{sec:examples}.

%
%
\section{Persistency of Excitation (PE) in Native Spaces} 
\label{sec:PEnative}
The basic error estimate for the Galerkin approximation described in Section \ref{sec:galerkinapprox} can be refined in several ways. In this section, we show how introducing priors on $G$, which are certain assumptions that enforce restrictions or constraints on the unknown function, can yield improved error estimates. In analogy to the case of parametric estimation in Euclidean space, we introduce a persistency of excitation condition for flows over a manifold. The PE condition can be used to derive alternative terms of an error bound on Galerkin approximations. Define the  operator 
\[
T_\phi(t,t+\Delta):=\int_{t}^{t+\Delta}\Ev_{\phi(\tau)}^*\Ev_{\phi(\tau)}\nu(d\tau).
\]

\subsection{A New Persistence of Excitation Condition} 
\label{sec:newPE}
In \cite{gpk2020pe}, we say that a persistence of excitation condition holds  over the closed subspace $\sV\subseteq \sH$   if there exist constants $\gamma_1,\gamma_2,\Delta>0$ such that 
\begin{align*}
    \gamma_1 \|g\|_{\sV}^2
    \leq &\int_t^{t+\Delta}\left \langle\Ev_{\phi(\tau)}^*\Ev_{\phi(\tau)}g,g \right \rangle_{\sH} \nu(d\tau) \leq \gamma_2 \|g\|_{\sV}^2 ,
    \end{align*} 
    or in other words
 \begin{align}
    \gamma_1 \|g\|_{\sV}^2 \leq &\left \langle T_\phi(t,t+\Delta)g,g \right \rangle_{\sH} \leq \gamma_2 \|g\|_{\sV}^2, \label{eq:PE}
\end{align}
for all $t\in \RR^+$ and $g\in \sV$. Note that the above PE condition uses the operator as given in Equation \ref{eq:bochnerT}.
We analyze two different cases below:
\begin{enumerate}
    \item[(1)] The space $\sV\subseteq \sH$ is the native space $\sH_S$ generated by a subset $S\subseteq X$, 
    \begin{align*}
        \sV:=\sH_S:=\overline{\text{span}\{ \knl_{x} \ | \ x \in S \}}.
    \end{align*}
    Note that in this case $\sV := \sH_S$ equipped with the norm it inherits as a closed subspace of $\sH$.
    \item[(2)] The space $\sV$ is selected to be the closed subspace $A^s(\sH)$ that is defined in terms of a fixed, compact, self-adjoint, positive operator $T^s$ and its spectral decomposition. 
\end{enumerate}

\subsection{Persistency in $\sV:=\sH_S$ with $S \subseteq X$}

Note that, if the above PE condition in Equation \ref{eq:PE} holds, we obtain  upper and lower bounds on $T_\phi(0,t)+\gamma I$.  For simplicity, suppose  that $t=m\Delta$ for some integer $m\in \NN$. Then, if the PE condition holds for $\sV := \sH_S$, we know that 
\[
(\gamma_1 m + \gamma)\|g\|_{\sH}^2 \leq 
\left \langle(T_\phi(0,t)+\gamma I)g,g \right \rangle_{\sH} \leq (\gamma_2 m + \gamma) \|g\|_{\sH}^2 
\]
for all $g\in \sH_S$. This means that for all $g\in \sH_S$ we have the upper bound 
\[
\|(T_\phi(0,t)+\gamma I)^{-1}g\|_{\sH} \leq \frac{1}{\gamma_1 m + \gamma}\|g\|_{\sH}.
\]

\subsubsection{The Optimal Regressor $\hat{g}(t,\cdot)$ in $\sH$}
\label{sec:optimalH}
We can use the above bound to find an error bound for the best offline estimate $\hat{g}(t,\cdot) \in \sH$. Suppose that $\Pi_S:\sH \rightarrow \sH_S$ is the $\sH$-orthogonal projection onto the closed subspace $\sH_S$. We set 
\begin{align*}
    \Delta G&:=G-\Pi_S G=(I-\Pi_S)G, \\
    y(t)&=\Ev_{\phi(t)}G=(\Pi_S G)(\phi(t))+ \Delta G(\phi(t)).
\end{align*}
We define the error between the offline estimate $\hat{g}(t,\cdot)$ and $\Pi_S G$ to be $\tilde{g}(t,\cdot):=\hat{g}(t,\cdot)-\Pi_S G$. It follows that 
\begin{align*}
    &\tilde{g}(t,\cdot)=\left (T_\phi(0,t)+\gamma I \right )^{-1}T_\phi(0,t)G-\Pi_S G \\
    &=\left (T_\phi(0,t)+\gamma I \right )^{-1}T_\phi(0,t)\left (\Pi_S G+ \Delta G\right ) -\Pi_S G \\
    &=\left (T_\phi(0,t)+\gamma I \right )^{-1}
    \left ( \left (T_\phi(0,t)+\gamma I \right )\Pi_S G -\gamma \Pi_S G\right ) 
    \\ &\hspace*{.75in} + \left (T_\phi(0,t)+\gamma I \right )^{-1}T_\phi(0,t)\Delta G 
    -\Pi_S G\\
    &=\left (T_\phi(0,t)+\gamma I \right )^{-1}
    \biggl ( \Pi_S\left (T_\phi(0,t)\Delta G -\gamma  G \right ) \\
    &\hspace*{1.25in}+ (I-\Pi_S)T_\phi(0,t)\Delta G 
    \biggr ).
\end{align*}
Now we apply the bound on $(T_\phi(0,t)+\gamma I)^{-1}$, and we get
\begin{align*}
    \|&\tilde{g}(t,\cdot)\|_{\sH} \leq \|\left (T_\phi(0,t)+\gamma I \right )^{-1}\Pi_S T_\phi(0,t)\Delta G\|_{\sH} \\
    & \hspace*{.3in}+\gamma \| \left (T_\phi(0,t)+\gamma I \right )^{-1}\Pi_S G\|_{\sH} \\
    &\hspace*{.3in}  
    +\| \left (T_\phi(0,t)+\gamma I \right )^{-1}\| \|(I-\Pi_S)T_\phi(0,t)\Delta G\|_{\sH}
    \\
    &\leq \frac{1}{\gamma_1 m +\gamma} \|T_\phi(0,t)\| \| (I-\Pi_S)G\|_{\sH} \\
    &\hspace*{.25in} + \frac{\gamma}{\gamma_1 m + \gamma}\|\Pi_S G\|_{\sH} 
    + \frac{1}{\gamma}\|T_\phi(0,t)\| \|(I-\Pi_S)G\|_{\sH}.
\end{align*}
But, by virtue of $\sqrt{\knl(x,x)}\leq \bar{\knl}$, we know that $\|T_\phi(0,t)\|\leq \bar{\knl}^2 \nu([0,t])$. Assuming that $\nu$ is Lebesgue measure on $\RR^+$ and $t=m\Delta$, we obtain
\begin{align}
    \|\hat{g}(t,\cdot)-\Pi_S G\|_{\sH} &\leq \frac{\bar{\knl}^2 m \Delta}{\gamma_1 m + \gamma} \|(I-\Pi_S)G\|_{\sH} \notag \\
    &\hspace*{.25in}+ \frac{\gamma}{\gamma_1 m + \gamma} \|\Pi_S G\|_{\sH} \notag \\
    &+ \frac{1}{\gamma} \bar{\knl}^2m\Delta \|(I-\Pi_S)G\|_{\sH}. \label{eq:HSbound}
\end{align}
 In particular, if we have $(I-\Pi_S)G=0$, we  conclude that 
\[
\underset{t\rightarrow \infty}{\text{lim }}\|\hat{g}(t,\cdot)-\Pi_SG\|_{\sH} =0. 
\]
\subsubsection{The Optimal Regressor in $\sH_S$}
\label{sec:optimalHS}
Above we characterized the optimal regressor $\hat{g}(t,\cdot) \in \sH$ when the subspace $\sH_S$ is PE. It is also possible to pose the original regression problem in $\sH_S$ and seek the optimal regressor $\hat{g}_S(t\cdot) \in \sH_S$ when $\sH_S$ is PE.
In this case we seek the approximation  $\hat{g}_S(t,\cdot)\in \sH_S$ that satisfies the equation
    \begin{align*}
        \langle(T_\phi(0,t)+ \gamma I)\hat{g}_S(t,\cdot),h_S\rangle_\sH = \langle T_\phi(0,t)G,h_S\rangle_\sH \quad \text{ for all } h_S\in \sH_S.
    \end{align*}
    Then the solution $\hat{g}_S(t,\cdot)$ can also be written as
    \[
    \hat{g}_S(t,\cdot)= \left (\Pi_S (T_\phi(0,t)+\gamma I)\Pi_S \right )^{-1}\Pi_S T_\phi(0,t) G.
    \]
We can apply the PE condition \[
    \langle \Pi_S(T_\phi(0,t) +\gamma I)\Pi_S g,g\rangle_\sH \geq (\gamma_1 m +\gamma)\|g\|^2_\sH
    \]
    for all $g \in \sH_S$. This gives an upper bound
    \[
    \|(\Pi_S(T_\phi(0,t)+ \gamma I)\Pi_S)^{-1} g\|_\sH \leq \frac{1}{\gamma_1 m +\gamma} \|g\|_\sH
    \]
    for all $g\in \sH_S$. In this case, $\hat{g}_S(t,\cdot) \in \sH_S$, so that  the error $\tilde{g}_S(t,\cdot) = \hat{g}_S(t,\cdot) - \Pi_S G =  \Pi_S \hat{g}_S(t,\cdot) - \Pi_S G $ can then be expressed as follows \[
    \tilde{g}_S(t,\cdot) = (\Pi_S( T_\phi(0,t)+\gamma I)\Pi_S)^{-1}[\Pi_S T_\phi(0,t)(I-\Pi_S)G - \gamma \Pi_S G].
    \]
    We can bound each of the two terms on the right hand side of the equality by writing
    \[
    \|(\Pi_S( T_\phi(0,t)+\gamma I)\Pi_S)^{-1}[\Pi_S T_\phi(0,t)(I-\Pi_S)G]\|_\sH \leq \frac{ \knl^2 m \Delta}{\gamma_1 m + \gamma} \|(I-\Pi_S)G\|_\sH,
    \]
    \[
    \|(\Pi_S( T_\phi(0,t)+\gamma I)\Pi_S)^{-1}[\gamma \Pi_S G] \|_\sH\leq \frac{\gamma }{\gamma_1 m + \gamma} \|G\|_\sH.
    \]
We now have the error bound
\begin{align}
    \|\hat{g}_S(t,\cdot)-\Pi_S G\|_{\sH} &\leq \frac{\bar{\knl}^2 m \Delta}{\gamma_1 m + \gamma} \|(I-\Pi_S)G\|_{\sH} \notag \\
    &\hspace*{.25in}+ \frac{\gamma}{\gamma_1 m + \gamma} \|\Pi_S G\|_{\sH}. \label{eq:galerkinBound}
\end{align}

\underline{Observations:}
\begin{enumerate}
    \item 
    {The analysis in Sections \ref{sec:optimalH} and \ref{sec:optimalHS} shows that the error in the optimal offline estimate can be controlled by a PE condition satisfied by  the operator $T_\phi(t,t+\Delta)$. The PE condition establishes that the best offline regressor estimate $\hat{g}(t,\cdot)$ of $G\in \sH_S$ converges to $G$.} 
    \item Note that the bound on $\tilde{g}$ given by Equation \ref{eq:HSbound} consists of three terms that each behave differently as $m \to \infty$. The first term converges to a constant proportional to the projection error $(I-\Pi_S)G$ as $m \to \infty$. The second term decays to zero as $m \to \infty$. However, the last term, referred to as the drift term, grows indefinitely as $m \to \infty$. If we seek to compute the optimal estimate $\hat{g}(t,\cdot) \in \sH$ via continuous regression when only $\sH_S$ is PE, this drift term results. The primary issue is that we cannot apply the PE condition
    \[
    \|(T_\phi(0,t)+\gamma I)^{-1}g\|_{\sH} \leq \frac{1}{\gamma_1 m + \gamma}\|g\|_{\sH}.
    \]
    on $\Delta G$ because $\Delta G \notin \sH_S$. This problem is addressed by seeking the optimal regressor $\hat{g}_S(t,\cdot) \in \sH_S$ that filters out this drift term.  This error $\tilde{g}_S(t,\cdot)$ is bounded by only two terms as given in Equation \ref{eq:galerkinBound} where the first converges to a constant proportional to $(I-\Pi_S)G$ and the second decays to zero as $m \to \infty$.

    \item In a typical situation, in applications to finite dimensional approximations, it is frequently the case that $S$ consists of a finite number of samples,
    $
     S:=\{\xi_{1},\ldots,\xi_{N} \}
    $, and then 
    \[
     \sH_S:=\sH_N:=\text{span}\{ \knl_{\xi_i}\ | \ 1\leq i\leq N \}.
     \] 
     In this case, the error between the best offline estimate $\hat{g}(t,\cdot)$ and the projection $\Pi_N G$ can be bounded above by the rate of convergence of $(I-\Pi_N)G$ to zero. This is carried out in detail in the example in Section \ref{sec:examples}. 
     
     \item Since the operator $T_\phi(t,\Delta)$ is compact, if the PE condition in Equation \ref{eq:PE} holds for $\sV:=\sH_S$, it must be the case that $\sH_S$ is finite dimensional. Otherwise, the PE condition would imply that the compact operator $T_\phi(t,\Delta): \sH \to \sH$ has a bounded inverse, which is impossible on an infinite dimensional space $\sH_S$. It follows that the primary application of case (1) will be to understand convergence of the optimal regressor when finite dimensional subspaces of approximants are PE.

\end{enumerate}

\subsubsection{Approximation: Method (1)}
\label{sec:approxMethod1}
     The exact optimal solution of the regression problem in continuous time, given by $\hat{g}(t,\cdot)=(T_\phi(0,t)+\gamma I)^{-1}T_\phi(0,t)G$, or correspondingly $\hat{g}_S(t,\cdot)$ defines a function of time and space, $\hat{g}(t,x)$ for $t\geq 0$ and $x\in X$. Practical implementations and algorithms must employ approximations of the exact regression solution. As in methods for parametric estimation in Euclidean spaces, recursive methods to approximate the solution of this problem are often used.  We have studied one recursive method for the problem in this paper in \cite{kepler}. Here, and in the numerical examples in Section \ref{sec:examples}, we comment on implementations of offline approximations.  Let $\hat{g}_N(t,\cdot)$ be the approximation of either $\hat{g}(t,\cdot) \in \sH$ or $\hat{g}_S(t,\cdot) \in \sH_S$. When we define the approximation $\hat{g}_N(t,x):=\sum_{j=1}^N \knl_{\xi_{N,j}}(x)\alpha_{N,j}(t)$ using the continuous regression error functional, we obtain an equation that has the form
\begin{align}
     \sum_{j=1}^N \biggl (\int_0^t \knl(\xi_{N,j},\phi(\tau))
     \knl(\phi(\tau),\xi_{N,i}) \nu(d\tau) &+  \gamma \knl(\xi_{N,j},\xi_{N,i}) \biggr ) \alpha_{N,j}(t) \notag \\
     &=  \int_0^t \knl(\xi_{N,i},\phi(\tau))y(\tau)\nu(d\tau). \label{eq:approxN1}
\end{align}
Note also that the optimal estimate in the continuous case requires integrating the kernel functions along the orbit, which ordinarily cannot be directly calculated in closed form. Consequently, the approximation requires approximating the integral term. For a function $\rho:[0,t] \to \RR$, a general form of a quadrature rule builds the approximation 
\[
\int \rho(\tau)d(\tau) \approx \sum_{k=1}^{N_w} w_k \rho(\tau_k),
\]
where $N_w$ is the number of quadrature points, $\{w_k\}_{k=1}^{N_w}$ are the quadrature weights, and $\{\tau_k\}_{k=1}^{N_w}$ are the quadrature points.  The approximation of the integral can be determined from a multitude of different quadrature techniques. Standard examples include the trapezoidal rule, Simpsons rule, or Gaussian quadratures. Here we use a particularly simple quadrature rule. Recall the above definitions of $\TT=[0,t]$, the subintervals $\TT_{N,k}$, and the quadrature points $\xi_{N,k}:=\phi(t_{N,k})$ for $t_{N,k}\in \TT_{N,k}$ from the proof of Theorem \ref{th:Tphi} where the weights $\{w_k\}_{k=1}^{N_w}$ are given by the time intervals $\{t_{N,k+1}-t_{N,k}\}_{k=1}^{N_w}$. One form of approximating the integral via a one point quadrature rule generates the equation 
\begin{align}
     \sum_{j=1}^N \biggl (
     \sum_{k=1}^{N_w} w_k &\knl(\xi_{N,j},\xi_{N,k})
     \knl(\xi_{N,k},\xi_{N,i})+  \gamma \knl(\xi_{N,j},\xi_{N,i}) \biggr )  \alpha_{N,j}(t)\notag \\
     &=   \sum_{k=1}^{N_w} w_k \knl(\xi_{N,i},\xi_{N,k})y(t_{N,k}). \label{eq:quadApprox}
\end{align}

We end this section with the pseudo-code implementation in Algorithm \ref{alg:approx1} to explicitly determine the estimate from this approximation method. Note that the initialization steps require that we select samples that have sufficient distance between one another ensuring the calculation of stable numerical estimates. \cite{powell2022koopman,wendland}
\begin{algorithm}
\caption{\centering Building an Estimate of the Continuous Regressor Using Quadrature Approximations }
\label{alg:approx1}
\begin{algorithmic}
\STATE{\textbf{input}: kernel function $\knl$, kernel hyperparameters $\beta$, regularization parameter $\gamma$, and desired kernel separation $\eta$}
\STATE{\textbf{select samples for kernel centers}: Collect data $\{t_i,x_i,y_i\}_{i=1}^M$}
\STATE{$S_N = \xi_{N,1} = x_1$, $N=1$}
\STATE{Select placement of kernel centers}
\FOR{$x_i$ in $\{x_i\}_{i=2}^M$}
\STATE{if $\|\xi_{N,j}-x_i\| > \eta, \quad \forall \xi_{N,j} \in S_N $}
\STATE{Add sample $x_i$ to the set of centers}
\STATE{$S_N = S_N \cup \{x_i\}$}
\STATE{Add sample time $t_i$ to the set of indexed times}
\STATE{$t_{N,i} = t_i$, $\TT_{N+1} = \TT_{N}\cup \{t_{N,i}\}$, $N = N+1$}
\ENDFOR
\STATE{Construct Kernel Matrix}
\STATE{$\KK(S_N,S_N)_{i,j} = [\mathfrak{K}(\xi_{N,i},\xi_{N,j}))]$}
\STATE{Determine Weighting Matrix}
\STATE{$\bm{W} = \text{diag}(w_k), \quad w_k = t_{k+1}-t_k$}
\STATE{Construct Output Vector}
\STATE{$\bm{y} = \{y_1,..y_N\}, \quad y_i = y(\xi_{N,i}) = y(\phi(t_{N,i}))$}
\STATE{Calculate coefficients $\bm{\alpha_N} = \{\alpha_{N,1},...,\alpha_{N,N}\}^T$ of estimate}
\STATE{$\bm{\alpha_N} = \bigl (\KK(S_N,S_N)\bm{W}\KK(S_N,S_N) +\gamma \KK(S_N,S_N)\bigr )^{-1} (\KK(S_N,S_N)\bm{W} \bm{y}$)}
\RETURN Estimate $\hat{g}_N(t,\cdot) = \sum_{j=1}^N \alpha_{N,j} \knl(\xi_{N,j},\cdot)$ .
\end{algorithmic}
\end{algorithm}

\subsubsection{Approximation: Method (2)}
\label{sec:approxMethod2}
The system of Equations \ref{eq:quadApprox} above requires the introduction of quadratures over $[0,t]$. In this section, we introduce a second method of approximation that eliminates the calculation of quadratures.  The coefficients $\{\alpha_{N,i}\}_{i=1}^N$ of the second method of approximation satisfy the following equation
\begin{align}
     \sum_{j=1}^N \int_0^t \knl(\xi_{N,j},\phi(\tau))
     &\knl(\phi(\tau),\xi_{N,k})\alpha_{N,j}(\tau) \nu(d\tau) \notag\\
     & +  \gamma \knl(\xi_{N,j},\xi_{N,i}) \alpha_{N,j}(t) \notag \\
     &=  \int_0^t \knl(\xi_{N,i},\phi(\tau))y(\tau)\nu(d\tau). \label{eq:approxN2}
\end{align}
In the above equation the unknown coefficients $\{\alpha_{N,i}(t)\}_{i=1}^N$ are now inside the integrand and are integrated along the orbit. Taking the time derivative of Equation \ref{eq:approxN2} we can get an evolution equation for the coefficients $\{\alpha_{N,i}(t)\}_{i=1}^N$. It is given by
\begin{align}
      \notag\\
    \dot{\alpha}_{N,j}(t) 
     = (\gamma \knl(\xi_{N,j},\xi_{N,i}))^{-1}
     \notag \\
     \times \biggl ( \knl(\xi_{N,j},\phi(t))y(t)
     - \sum_{j=1}^N &\knl(\xi_{N,j},\phi(t))
     \knl(\phi(t),\xi_{N,i})\alpha_{N,j}(t)\biggr ) . \label{eq:approxNevolution}
\end{align}
As opposed to the previous approximation method, estimates constructed according to Equation \ref{eq:approxNevolution} evolve continuously according to the ODE given by Equation \ref{eq:approxNevolution}, rather than approximations generated by Equation \ref{eq:approxN1}. As opposed to the previous offline optimization problem, this approximation method could, in principle, generate and update estimates in real-time. We emphasize that the theory presented in this paper only applies to Method (1), and we leave the theoretical study of Method (2) for a future paper. However, for completeness, we examine the performance and convergence behavior of both methods in the numerical results of this study. 

Like the Method (1), we present Algorithm \ref{alg:approx2} to outline the steps needed to implement this approximation method.
\begin{algorithm}
\caption{\centering Building an Estimate of the Continuous Regression Estimate Using Approximation Method (2)}
\label{alg:approx2}
\begin{algorithmic}
\STATE{\textbf{input}: kernel function $\knl$, kernel hyperparameters $\beta$, regularization parameter $\gamma$, and desired kernel separation $\eta$}
\STATE{\textbf{kernel placement steps}: Run system and collect data $\{t_i,x_i,y_i\}_{i=1}^M$}
\STATE{$S_N = \xi_{N,1} = x_1$, $N=1$}
\STATE{Select placement of kernel centers from initial orbit}
\FOR{$x_i$ in $\{x_i\}_{i=2}^M$}
\STATE{if $\|\xi_{N,j}-x_i\| > \eta, \quad \forall \xi_{N,j} \in S_N $}
\STATE{Add sample $x_i$ to the set of centers}
\STATE{$S_{N+1} = S_{N}\cup \{x_i\}$}
\STATE{Add sample time $t_i$ to the set of indexed times}
\STATE{$t_{N,i} = t_i$, $\TT_{N+1} = \TT_{N}\cup \{t_{N,i}\}$, $N = N+1$}
\ENDFOR
\STATE{Construct Kernel Matrix}
\STATE{$\KK(S_N,S_N)_{i,j} = [\mathfrak{K}(\xi_{N,i},\xi_{N,j}))]$}
\STATE{Choose initial conditions}
\STATE{$\phi(0) = x_0$}
\STATE{Denote $\bm{\alpha_N} = \{\alpha_{N,1},...,\alpha_{N,N}\}^T$ coefficient vector of estimate}

\STATE{Denote $\knl(S_N,\cdot) = \{k(\xi_{N,1},\cdot),...,k(\xi_{N,N},\cdot)\}^T$ vector of kernel functions}
\STATE{Numerically integrate ODEs}
\STATE{$\dot{\phi}(t) = f(\phi(t))$}
\STATE{$\dot{\bm{\alpha}}_{N}(t) 
     = (\gamma \KK(S_N,S_N))^{-1}\biggl ( \knl(S_N,\phi(t))y(t)
     - \knl(S_N,\phi(t))
     \knl(S_N,\phi(t))^T\bm{\alpha}_N(t)\biggr ) $ }
\RETURN Estimate $\hat{g}_N(t,\cdot) = \sum_{j=1}^N \alpha_{N,j}(t) \mathfrak{K}_{N,j}(\cdot)$ .
\end{algorithmic}
\end{algorithm}

\subsection{Learning Theory and the Regression Estimate} 
\label{sec:learningTheory}
In this section, we give an expanded discussion of the similarities and differences between the approximations described in Sections \ref{sec:approxMethod1} and \ref{sec:approxMethod2} in this paper and related techniques in distribution-free learning theory, statistical learning theory, and machine learning theory. For the most part, these learning theory approaches focus on estimates generated from samples of discrete, independent and identically distributed (IID) stochastic systems. See \cite{gyorfi,williams} for popular summaries of the state-of-the-art in these fields.
Learning theory in general \cite{vapnik1999overview} is concerned with a number of distinct problems including pattern recognition, classification, and function estimation. The learning problem for function estimation involves approximating a mapping $G$ from a set of inputs $x \in X$ to elements $y$ in an output space $Y$. It is commonly assumed that the data is a collection of $M$ noisy sample pairs $\{z_i\}_i^M =\{x_i,y_i\}_i^M \subset Z = X \times Y$ that are generated from a discrete IID stochastic process defined by the probability measure $\mu$ on $Z$. Ideally, optimal estimates of $G$ are defined to be minimizers of the functional $E_\mu$, commonly referred to as the expected risk,

\begin{equation*}
    E_\mu(g):=\int |y-g(x)|^2 \mu(dz),
\end{equation*}
where $\mu$ is a joint measure on the sample space $Z$. Note that the measure $\mu$ can be rewritten, $\mu(dz) = \mu(dy,dx) = \mu_Y(dy|x)\mu_X(dx)$, where $\mu_Y(dy|x)$ is the conditional measure on $Y$ given $x \in X$ and $\mu_X$ is the marginal measure on the input space $X$.  In principle, the ideal minimizer of $E_\mu$ is given by $$G_\mu(x) = \int y\mu_Y(dy|x)$$ with $G_\mu$ referred to as the regressor function. However, the measures $\mu$, $\mu_Y$, and $\mu_X$ are generally unknown, and the ideal solution, $G_\mu$ above cannot be computed in practice. It is this reason that the above problem is said to define a type of distribution-free learning problem \cite{gyorfi}. 

Since the regressor $G_\mu$ cannot be computed in general, standard approaches in machine learning theory replace the error functional above with its regularized, discrete counterpart
\begin{equation*}
    E_M(g) := \frac{1}{M}\sum_{i=1}^M |y_i - g(x_i)|^2 +\gamma \|g\|^2_\mathcal{U} 
\end{equation*}
that is defined in terms of the samples $\{(x_i,y_i\}_{i=1}^M$ of a discrete IID stochastic process. Here $\gamma$ is the regularization parameter and $\mathcal{U}$ is a space of functions that has a norm that measures smoothness. When some finite dimensional space $\sH_N = \text{span}\{\psi_j\ | \ 1\leq j\leq N\}$ is used to construct approximations, the method of empirical risk minimization (ERM) seeks the function $\hat{g}_{N,M}(\cdot) = \sum_{j=1}^N {\alpha}_{N,j}\psi_j(\cdot)$ that is the minimizer
\[
\hat{g}_{N,M} = \min_{g\in\sH_N}E_M(g_N)
\]
Note that the minimizer $\hat{g}_{N,M}$ depends on the number of samples $M$ and the number of basis functions $N$. The convergence of $\hat{g}_{N,M} \to G$ as $N$ and $M$ increase is a well-studied topic, certainly one of the most well-known in learning theory. Again, see \cite{gyorfi,williams} for a complete description of the myriad of approaches to this problem.
The relationship of the approach in this paper to the standard learning problem can be made more precise by assuming that the basis is taken to be $\sH_N = \text{span}\{\knl_{\xi_{N,i}}\ | \ 1\leq i\leq N\}$ that is the scattered basis as we use in Equations \ref{eq:approxN2} and \ref{eq:approxNevolution}. In this case, it is well-known that $\hat{g}_{N,M}(\cdot) = \sum_{j=1}^N {\alpha}_{N,j} \knl_{\xi_N,j} $ where the coefficients satisfy
\begin{equation}
\sum_{j=1}^N\biggl(\sum_{i=1}^M \knl_{\xi_{N,j}}(x_i)\knl_{\xi_{N,k}}(x_i) + \gamma\langle \knl_{\xi_{N,j}},\knl_{\xi_{N,k}}\rangle_\mathcal{U}\biggr) {\alpha}_{N,j} = \sum_{i=1}^M y_i \knl_{\xi_{N,k}}(x_i)
\label{eq:standardApprox}
\end{equation}

These equations should be carefully compared to Equations \ref{eq:approxN1} and \ref{eq:approxN2}. In Equation \ref{eq:standardApprox} the inner summation is over the samples from an IID process see \cite{devore2004mathematical,temlyakov2018multivariate}. For the regression problem in continuous time in Equation \ref{eq:approxN1}, the inner summation above is replaced with an integration in time along a trajectory. We see that the use of one point quadrature rule in time, which yields Equation \ref{alg:approx1}, generates a set of algebraic equations that have a similar structure to that which arises in learning theory for discrete stochastic processes in Equation \ref{eq:standardApprox}. In fact, if the sample times are uniformly distributed, the weights of integration are constant and can be cancelled in Equation \ref{eq:approxN1}. In such a case, the two sets of equations have identical form. Although the form of the equations is the same, the error analysis for the two cases differs substantially. One significant difference is that, in the typical learning theory scenarios, it is assumed that samples are dense in $X$, while this is a rather special case for deterministic dynamical systems. For dynamical systems, the set over which the error analysis is performed is typically unknown. The set of samples along a trajectory can be dense in a very irregular set. This fact is emphasized in the numerical example in Section \ref{sec:examples}. Additionally, the error analysis for Equation \ref{eq:approxN1} relies on a PE condition that is not part of the stochastic framework. The error analysis of Equation \ref{eq:standardApprox} usually results from taking the expectation. For deterministic dynamical systems, however, there is no definition of expectation. This work instead considers when the inputs are generated along a trajectory $t \mapsto \phi(t)$ governed by some underlying, generally unknown evolution.   For a more in-depth discussion of learning theory for unknown discrete samples, see the work of Cucker and Zhou in \cite{cucker}  or Devito et. al. in \cite{belkindevito}. 

\subsection{Subspaces $\sV$ Chosen as a Spectral Space $A^s$}
The case studied above suffices to derive rates of convergence of approximations in finite dimensional spaces $\sH_S$ where $S$ is a finite set of points.  From a practical point of view, the results apply to many important cases that can be implemented.  However, from a theoretical point of view, we would like to be able to identify a closed subspace $\sV$ that is ``as large as possible''   in the definition in Equation \ref{eq:PE}.  It is of interest therefore to  find a space $\sV\subseteq \sH$ that is infinite dimensional and satisfies the PE condition. We do this by introducing spectral approximation spaces $A^s$ associated with a fixed, compact, self-adjoint operator $T: \sH \to \sH$.  By the spectral theorem for compact, self-adjoint operators, this means that the operator can be expressed as 
\[
T h:=\sum_{k=1}^\infty \lambda_k \langle h, h_k \rangle_\sH h_k, 
\]
where $\{\lambda_k\}_{k\in \NN}$ is the sequence of eigenvalues arranged in nonincreasing order and repeated as needed for  multiplicity, and $\{h_k\}_{k\in \NN}\subset \sH$ is a  corresponding $\sH$-orthonormal collection of eigenfunctions. The only possible accumulation point of the eigenvalues is zero. In the following we always assume that $\lambda_k\rightarrow 0$ as $k\rightarrow \infty$, since if the sum above terminates after a finite number of terms, all the approximation spaces introduced below  degenerate and are equivalent. We also assume that the kernel of $T$ is equal to $\{0\}$. By  \cite{belkindevito} Proposition 8,  the eigenvectors of $T$ span $\text{nullspace}(T)^\perp$, so in the case at hand $\{h_k\}_{k\in \NN}$ are an orthonormal basis for $\sH$.   

By virtue of the functional calculus for compact, self-adjoint operators, the operator $T^s$ is well-defined for all $s\geq 0$ by the expansion 
\[
T^sh:=\sum_{k=1}^\infty \lambda^s_k \langle h, h_k \rangle_\sH h_k, \quad \text{ strongly in } \sH.
\]
We define the spectral approximation space $A^s:=A^s(\sH)$ to be
\begin{align*}
    A^s:=\left \{
    h\in \sH\ \biggl | \ \|h\|_{A^s}<\infty
    \right \},
\end{align*}
where the norm is given by 
\begin{align*}
\left \| h\right \|^2_{A^s}:=
\sum_{k=1}^\infty (\lambda_k^{-s}|\langle h, h_k \rangle_\sH|)^2 : = \sum_{k=1}^\infty \lambda_k^{-2s}|\langle h, h_k \rangle_\sH|^2
\end{align*}
 The spaces $A^s$  have a long history and are closely related to approximation spaces. \cite{piestch81,devorelorentz}.  Since $\lambda_k\rightarrow 0$, the weight $\lambda_k^{-2s}$ grows as $k\rightarrow \infty$. The space $A^s$ consists of functions in $\sH$ whose generalized Fourier coefficients $\{\langle h,h_k\rangle_\sH \}_{k\in \NN}$converge faster than $\{\lambda_k^{-2s}\}_{k \in \NN}$ increases. It can be shown that these spaces are nested with $A^{s}\subseteq A^r$ whenever $r\leq s$. It should also be noted that $A^0=\sH$. So, in particular we have $A^{s+1} \subseteq A^{s} \subseteq \cdots \subset \sH$. In the language of approximation theory, the $A^{s}$ define a scale of spaces for $s\geq 0$ containing functions of increased (generalized) smoothness as $s$ increases.  
 
The following theorem provides the technical connection between the spaces $\sH$ and spectral space  $A^1$ in terms of the operator $T$. 
\begin{theorem}
\label{th:spectralspaces}
We have the equivalence 
\[
\|h\|^2_{A^s}\approx \langle T^{-s} h, T^{-s} h\rangle_\sH.
\]
 for all $h\in \text{domain}(T^{-s})\equiv A^s$ and $T^{s}:\sH\rightarrow A^s$ is an isometry.
\end{theorem}
\begin{proof}
Suppose $h\in \sH$. Then 
\[
\left \langle T^{-s} h, T^{-s}h\right \rangle_\sH = \sum_{k=1}^\infty
\lambda_k^{-2s}|\langle h,h_k\rangle_{\sH}|^2 = \|h\|^2_{A^s}.
\]
It is also immediate that 
$
\|T^{s}h\|_{A^s}=\|h\|_\sH 
$, so $T^{s}$ is an isometry from $\sH$ onto $A^{s}$. In particular $T:\sH\rightarrow A^1$ is an isometry. 
\end{proof}

\noindent In view of Theorem \ref{th:spectralspaces}, when the PE condition holds,  we have constants $\gamma_1, \gamma_2>0$ such that 
\[
\gamma_1 \|h\|^2_{A^1} \leq  \left \langle T_{\phi}(t,t+\Delta)h,h \right \rangle_{\sH}  \leq \gamma_2 \|h\|_{A^1}^2
\]
for all $h\in A^1\subseteq \sH$ and $t\in \RR^+$. Note that this equivalence holds uniformly for the family $\{T_\phi(t,t+\Delta t)\}_{t\in\RR^+}$ for all $t\in\RR^+$. This pair of inequalities can also be interpreted as the statement that $T_\phi(t,t+\Delta) \approx T$ on $A^1.$

 Now we return to the study of the error when we choose $\sV:=A^{1}$, and we seek the optimal $\hat{g}_\sV(t,\cdot) \in A^1$.  In analogy to Case 1, we set 
\begin{align}
    \Delta G&:=(I-\Pi_\sV)G, \notag \\
    y(t) &=(\Pi_\sV G)(\phi(t)) + \Delta G(\phi(t)),\\
    \tilde{g}_\sV(t,\cdot)&:=\hat{g}_\sV(t,\cdot)-\Pi_\sV G. \label{eq:newerror}
\end{align}
Following the same plan of attack as in Case 1, we obtain
\begin{align*}
    \tilde{g}_\sV(t,\cdot)=\left (\Pi_\sV(T_\phi(0,t)+\gamma I)\Pi_\sV \right )^{-1}\left (\Pi_\sV T_\phi(0,t) \Delta G -\gamma \Pi_\sV G \right )
\end{align*}

In this case, in contrast,  we can write
\begin{align*}
   \left \langle (T_\phi(0,t)+\gamma I)h,h \right \rangle_{\sH}
   & \geq \sum_{k=1}^m \left \langle T_\phi((k-1)\Delta,k\Delta)h,h\right \rangle_{\sH} \\
   &\hspace{2em} +\gamma \langle h,h\rangle_{\sH}\\
   & \geq
   (\gamma_1 m +\gamma) \|h\|_{A^1}^2 
\end{align*}
for all  $h\in A^1 \subseteq \sH$. 
In other words $(T_\phi(0,t)+\gamma I)^{-1}$ restricted to $A^1\subseteq \sH$ is a bounded linear operator that satisfies 
\begin{equation}
{\|(T_\phi(0,t)+\gamma I)^{-1} h\|_{\sV}\leq 
\frac{1}{\gamma_1 m+\gamma } \|h\|_\sH.} \label{eq:newerror1}
\end{equation}
By definition, $\Pi_\sV G\in A^1$ and $\Pi_\sV T_\phi(0,t)\Delta G\in A^1$. 
 Since $(\Pi_\sV T_\phi(t)\Delta G-\gamma \Pi_\sV G)\in A^1$, we can apply the bound in Equation \ref{eq:newerror1} to Equation \ref{eq:newerror}.  The remainder of the proof is unchanged and we conclude that, if $G\in A^1$, 
\[
\limsup_{t\rightarrow \infty} \|\hat{g}_\sV(t,\cdot)-\Pi_\sV G\|_\sH=0.
\]

\noindent \underline{Observations}:
\begin{enumerate}
    \item It should be emphasized that the operator $T_\phi(t,t+\Delta):\sH\rightarrow \sH$ is compact, as described in Theorem \ref{th:Tphi}. But when the PE condition holds with $\sV:=A^1$,  it is not compact as an operator $\sH\rightarrow A^1$. It is boundedly invertible as a map from $\sH\rightarrow A^1$. 
    \item Intuitively, the PE condition can be understood as a statement that the local approximation space defined over the small time-span $[\tau,\tau+\Delta]$ in terms of the operator $T_\phi(\tau,\tau+\Delta)$ is spectrally equivalent to the global approximation space defined over $[0,t]$ in terms of $T_\phi(0,t)$. 
\end{enumerate}

\section{Numerical Examples}
\label{sec:examples}
The error estimates above apply to quite general situations. Since some of our earlier works in \cite{pgk2020suff,kgp2019limit,powell2022koopman} have included numerical examples with  evolutions on compact manifolds, here we model  a trajectory that is dense in a complicated, unknown subset in $\RR^n$.  Consider the Lorenz system \begin{align*}
    \dot{x}(t)&=\sigma(y(t)-z(t)), \\
    \dot{y}(t)&=r(x(t)-y(t) - x(t)z(t)),\\
    \dot{z}(t)&=x(t)y(t)-bz(t)
\end{align*}
for $t\in \RR^+$. Figures \ref{fig:lorenzOrbitsXYZ} and \ref{fig:lorenzOrbitsXY}  illustrate orbits of the system for various initial conditions. Set $X=\RR^3$ and denote by $\phi(t):=\{x(t),y(t),z(t)\}^T$.
The complex nature of the trajectories of this system has been studied and  commented on so extensively that it is now understood as an exemplar of what chaos and complexity is, even in the popular press. 

\begin{figure}[ht!]
    \centering
    \begin{subfigure}[b]{0.49\textwidth}
    \includegraphics[width=\textwidth]{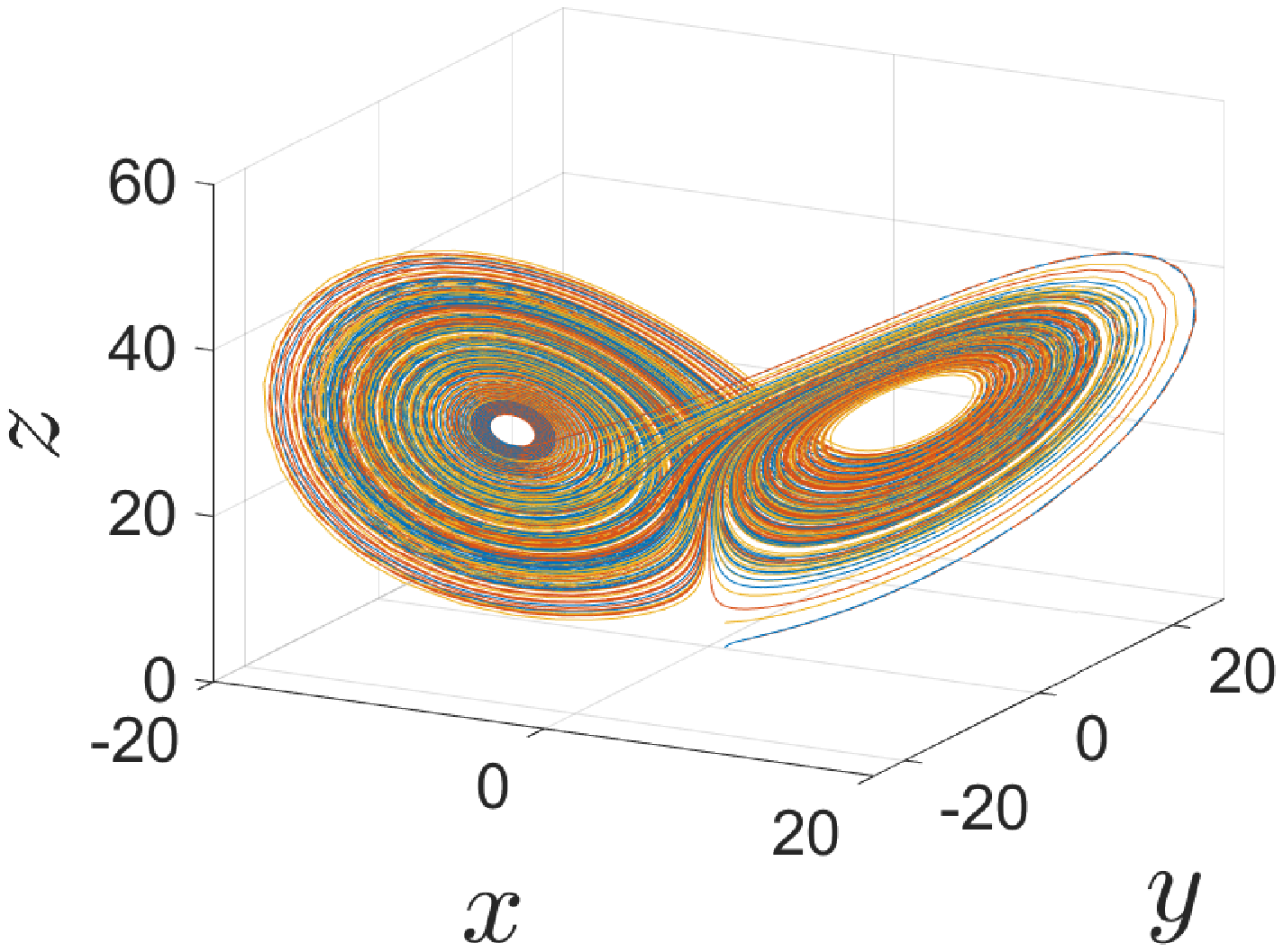}
    \caption{}
    \label{fig:lorenzOrbitsXYZ}
    \end{subfigure}
    \hfill
    \begin{subfigure}[b]{0.49\textwidth}
    \includegraphics[width=\textwidth]{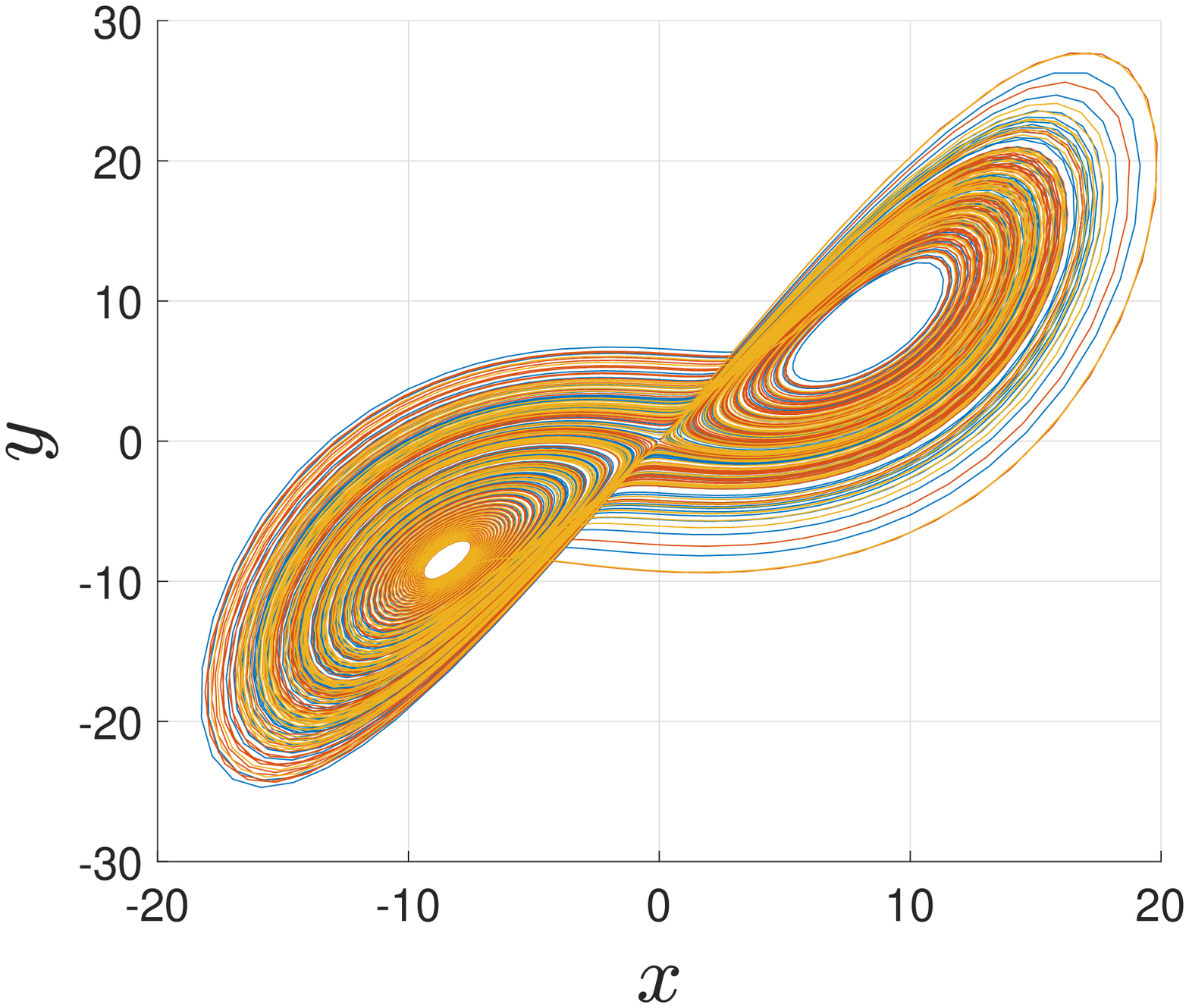}
    \caption{}
    \label{fig:lorenzOrbitsXY}
    \end{subfigure}
    \caption{The dynamics of the Lorenz system for several initial conditions (a) in three-dimensions (b) projected onto the $x-y$ plane.}
\end{figure}

There are a number of Lyapunov functions that have been introduced to study the long-term behavior of this system. One common choice is 
\[
V=rx^2 + \sigma y^2 + \sigma (z-2r)^2 \geq 0
\]
for all $(x,y,z)\in X$. Its derivative along trajectories is given by 
\[
\dot{V} = -2\sigma(rx^2 + y^2 + bz^2 -2brz ),
\]
which is negative outside of the compact set 
\[
\Omega:=\{(x,y,z)\in X \ | \ \dot{V}\geq 0 \}. 
\]
The fact that this set is compact follows by demonstrating that the boundary of the set $\{(x,y,z) \in X \ | \ \dot{V} = 0\}$ is an ellipse given by$$ 1 = \frac{x^2}{(\sqrt{br})^2}+\frac{y^2}{(\sqrt{b}r)^2}+ \frac{(z-r)^2}{r^2}.$$ Given that $\dot{V}$ is continuous and that points inside the ellipse satisfy the following inequality $$ 0 \geq rx^2 + y^2 +bz^2-2brz,$$  it is clear that the set $\Omega$ is the closure of the interior of the ellipse. Since $\Omega$ is compact and $V$ is continuous, the maximum $\bar{V}$ of $V$ over $\Omega$ is achieved, $\bar{V}=\max_{\phi \in \Omega}V(\phi)$. Define the dilation of the set $\Omega$ by some parameter $\epsilon>0$ to be
\[
\Omega_\epsilon:=\{\phi \in X \ | \ V(\phi)\leq \bar{V}+\epsilon \}.
\]
The set $\Omega_\epsilon$ is positive invariant. Any trajectory starting at $\phi_0\not \in \Omega$ is guaranteed to enter $\Omega_\epsilon$ in finite time and never leave this set. This means that for any initial condition $\phi_0\not \in \Omega$, the orbit $\Gamma(\phi_0):=\cup_{t\geq 0}\phi(t)$ is precompact, that is, $\overline{\Gamma(\phi_0)}$ is compact.  From standard results on dynamical systems \cite{walker}, it is known that the positive limit set $\omega^+(\phi_0)$  of a precompact trajectory $t\mapsto \phi(t)$, which is defined by 
\[
\omega^+(\phi_0):=\left \{ y\in X\ | \ \exists  t_k\rightarrow \infty \text{ such that } \phi(t_k)\rightarrow y \right \},
\]
is compact. In fact, we have 
\[
\phi(t)\rightarrow \omega^+(\phi_0)\subseteq \overline{\Gamma^+(\phi_0)}.
\]
Both $\omega^+(\phi_0)$ and $\overline{\Gamma^+(\phi_0)}$ are guaranteed to be compact sets, but they can be highly irregular. For the case at hand, where we study the Lorenz system,  this fact is well-known.

We want to use the results of this paper to understand what can be said about the regression problem in continuous time for this system, and we are particularly interested in what the error bounds imply for estimates of an observable function for this system. We would like to understand how the trajectory $t\mapsto \phi(t)$ affects convergence of approximations, and to determine in what spaces the continuous time regression problem converges. We can use either of the sets $\omega^+(\phi_0)$ or $\overline{\Gamma^+(\phi_0)}$  to study the convergence properties. 

In this paper,  we study the case when  the set $S=\overline{\Gamma^+(\phi_0)}$. We define different sets $S_N\subset S$ that have $N$ samples in $\Gamma^+(\phi_0)$, $S_N:=\{\xi_1,\ldots \xi_N\}:=\{\phi(t_1),\phi(t_2),\\ \ldots,\phi(t_N)\}$. We assume that these are nested, $S_{N}\subset S_{N+1}$.  Associated with $S_N$ we define the space of approximants in terms of a scattered basis with  $\sH_N:=\text{span}\{\knl_{\xi_1},\ldots, \knl_{\xi_N}\}$. These finite dimensional spaces of approximants are data driven: they are generated along a trajectory of the system. For each $N$, suppose that the PE condition holds for $\sH_N$. Whether or not the PE condition holds in the case that $\sV=\sH_N$ can be verified by conditions related to the visitation, or time of occupation, of the trajectory $\tau\mapsto \phi(\tau)$ in neighborhoods of the samples in $S_N$. See \cite{gpk2020pe,pgk2020center,pgk2020suff} for a discussion.   If the trajectory persistently excites the subspace $\sH_N$, we have the estimate from Equation \ref{eq:galerkinBound}
\begin{align*}
 \|\hat{g}_N(t,\cdot)&-\Pi_N G\|_{\sH_N}  \\ &\leq \left (\frac{\bar{\knl}^2 m \Delta}{\gamma_1 m + \gamma}\right )
 \|(I-\Pi_N)G\|_{\sH_N} + \frac{\gamma}{\gamma_1 m + \gamma} \|\Pi_N G\|_{\sH_N}. 
\end{align*}
Concrete estimates of the rate of convergence of this expression can be obtained using the power function $\calP_N(x)$ over the set $S=\overline{\Gamma^+(\phi_0)},$
\begin{equation}
    \calP_N(x):=|\knl(x,x)-\knl_N(x,x)| \quad \text{ for all } x\in S:= \overline{\Gamma^+(\phi_0)}
    \label{eq:powerFunction}
\end{equation}
where $\knl_N$ is the kernel that defines the native space $\sH_N$. It is well-known \cite{wendland,haasdonk} that the power function $\calP_N(x)$ enables the pointwise bound
\[
|((I-\Pi_N)h)(x)|\leq \calP_N(x) \|h\|_{\sH} \quad \text{ for all } x\in S:=\overline{\Gamma^+(\phi_0)}.
\]
This pointwise bound can be used to derive a corresponding bound on $\|(I-\Pi_N)h\|_\sH$, for smooth enough $h$. The details exceed the length of this brief paper and are given in \cite{begklpp2022} in a  different  application to approximation of Koopman operators, or this bound can be inferred from the proof of Theorem 11.23   in \cite{wendland}.  Ultimately, we obtain a bound
\[
\|(I-\Pi_N)h\|_{\sH} \lesssim  \left (\int_{S} |\calP_N(x)|^2 dx \right )^{1/2} \|h\|_\sH
\]
for all $h\in \sH$ that are smooth enough, which means that the optimal regression estimate in continuous time  satisfies 
\begin{align}
 \|&\hat{g}_N(t,\cdot)-\Pi_N G\|_{\sH_S}  \\ &\leq \left(\left (\frac{\bar{\knl}^2 m \Delta}{\gamma_1 m + \gamma} \right )
\|\calP_N\|_{L^2(S)}\right)\|G\|_{\sH_S} +  \frac{\gamma}{\gamma_1 m + \gamma}\|\Pi_N G\|_{\sH_S}. 
\label{eq:powerBound}
\end{align}
for all $G$ smooth enough. 

We begin with an assessment of the numerical implementation of Method (1). We illustrate the performance of an  estimate $\hat{g}_N(t,\cdot)$ defined in Equation \ref{eq:approxN1} of $ \Pi_N G$  generated over an orbit $\Gamma(\phi_0)$ of the Lorenz system. We only pose the regression problem over a projection of the orbit $\Gamma(\phi_0)$ onto the $x-y$ plane so that the results are easy to visualize. The function $G$ we estimate is given by
$$ G(x,y) = -10 \text{sin}\biggl( \frac{y}{10} \biggr) + \frac{(x+y)^3}{2000}+200. $$ and we choose the initial condition $\phi_0 = \{1,1,1\}^T$. In this example, we use the Matern-Sobolev kernel, 
\begin{equation*}
    \knl(\xi_{N,j},x) = \bigg(1+ \frac{\sqrt{3}\|x-\xi_{M,j}\|_2^2}{\beta }\bigg)e^{\big(-\frac{\sqrt{3}\|x-\
    \xi_{M,j}\|_2^2}{\beta }\big)}
\end{equation*}
with $\| \cdot \|_2$ the standard Euclidean norm over  $\RR^2$, $j \in {1,...,N}$, and the hyperparameter $\beta = 5$. In Figure \ref{fig:kernelEstimate}, the output $G(\phi(t))$ for $t \geq 0$ is represented by the red curve that hovers over the dynamics of the input orbit labeled by the black curve $\phi(t)$ in the $x - y$ plane. In this figure, 126 kernel centers are selected quasi-uniformly along the orbit with a separation distance of around 2 and the regularization parameter $\gamma = 0.1$.
\begin{figure}[ht!]
    \centering
    \includegraphics[width=0.9\textwidth]{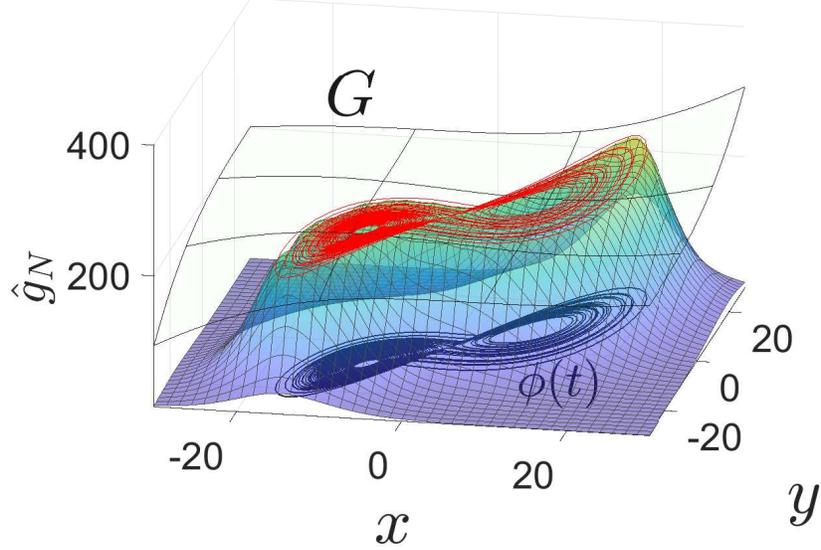}
    \caption{An illustration of the kernel estimate $\hat{g}_N(t,\cdot)$ of $G$  for $t = 200$ seconds defined over the orbit $\phi(t)$ starting at some initial conditions $\phi_0 = \{1,1,1\}^T$. The orbit is generated by projecting the Lorenz system dynamics onto the x-y plane. The output $G(\phi(t))$ for $t \geq 0$ is represented by the red curve hovering over the dynamics of the input orbit $\phi(t)$. The estimate $\hat{g}_N(t,\cdot)$ represented by the colored mesh minimizes its error from the projection of the true function $G$ (the curved surface) over the orbit $\phi(t)$.}
    \label{fig:kernelEstimate}
\end{figure}

From the figure, it is clear that the approximation $\hat{g}_N(t,\cdot)$ represented by the colored mesh yields, qualitatively speaking, a good estimate of the true function $G$ (the green surface) over the orbit.

When interpreting this result, it is important to keep several facts in mind. 

\noindent \underline{Observations:}
\begin{enumerate}
    \item 
The theory in this paper uses the   compact subset $S:=\overline{\Gamma(\phi_0)}$, whose regularity is not easy to characterize. The set $\overline{\Gamma(\phi_0)}$ defines  the space $\sH_S$ in which regression approximations in continuous time converge.  
\item 
The convergence of estimates is in the space $\sH_S$, which is an RKHS space of funnctions over the set $X$. Even though $S$ is quite irregular, the functions in $\sH_S$ are supported on the whole set $X$, not just $S$. This means that estimate is, in a sense, ``naturally extended'' to the whole state space $X$. In the case at hand, the set $S$ has zero Lebesgue measure.  Even though the trajectory or orbit  may not reach some points or subsets of $\RR^2$,  the function estimates are well-defined everywhere nonetheless. 
\end{enumerate}

\subsection{Example: Characteristics of Approximation Method (1)}
The next set of results examines the estimates of approximation Method (1) over orbits spanning  different intervals of time. Using the same underlying input dynamics, kernel function, hyperparameter $\beta$, regularization parameter $\gamma$, and unknown function $G$ from the previous results, each of the estimates are generated from an orbit starting at an initial condition $\phi_0 = \{1,1,1\}^T$. The centers are placed quasi-uniformly along the orbit with a separation distance of around 2. Figures \ref{fig:estimate10s-1} through \ref{fig:estimate200s-1} illustrate the estimates calculated over different spans of time. In Figure \ref{fig:estimate10s-1}, we can see that, even for small time intervals, the error between the estimate $\hat{g}_N(t,\cdot)$ and the true function $G$ begins to diminish significantly over the orbit. Additionally, there is significant decrease in the error over the trajectory as seen in Figures \ref{fig:estimate10s-1} and \ref{fig:estimate20s-1}. While the error continues to decrease as $t \to \infty$, it is evident the error reduction between Figures \ref{fig:estimate50s-1} and \ref{fig:estimate200s-1} occurs at a much smaller rate than the previous time intervals. This is a consequence of diminishing rate of reduction in the power  function $\calP_N(x) |_{x\in \overline{\Gamma(\phi_0)}}$ as $t \to \infty$ and more samples are collected.
\begin{figure}[ht!]
    \centering
    \begin{subfigure}[b]{0.47\textwidth}
    \includegraphics[width=\textwidth]{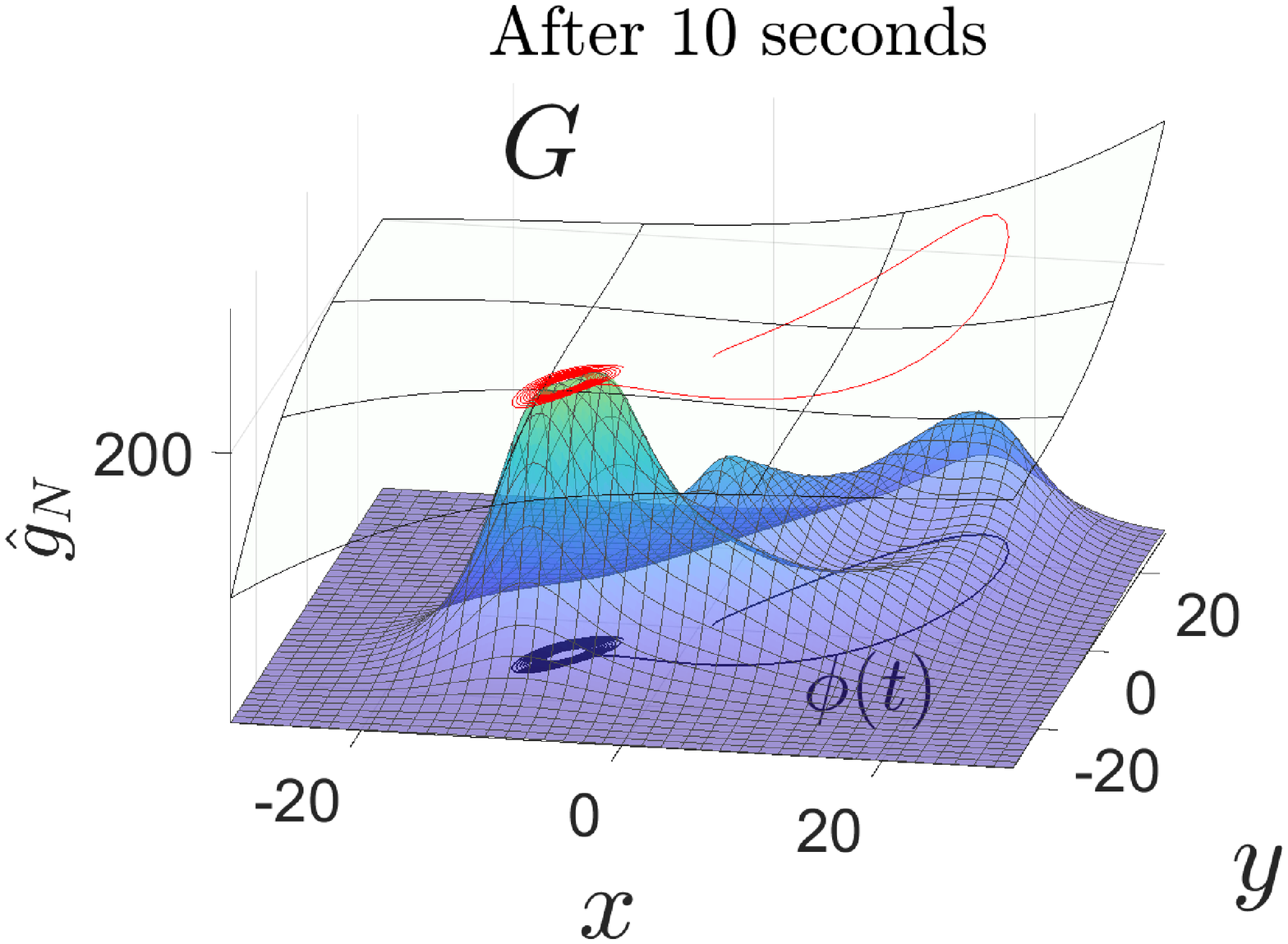}
    \caption{}
    \label{fig:estimate10s-1}
    \end{subfigure}
    \hfill
    \begin{subfigure}[b]{0.47\textwidth}
    \includegraphics[width=\textwidth]{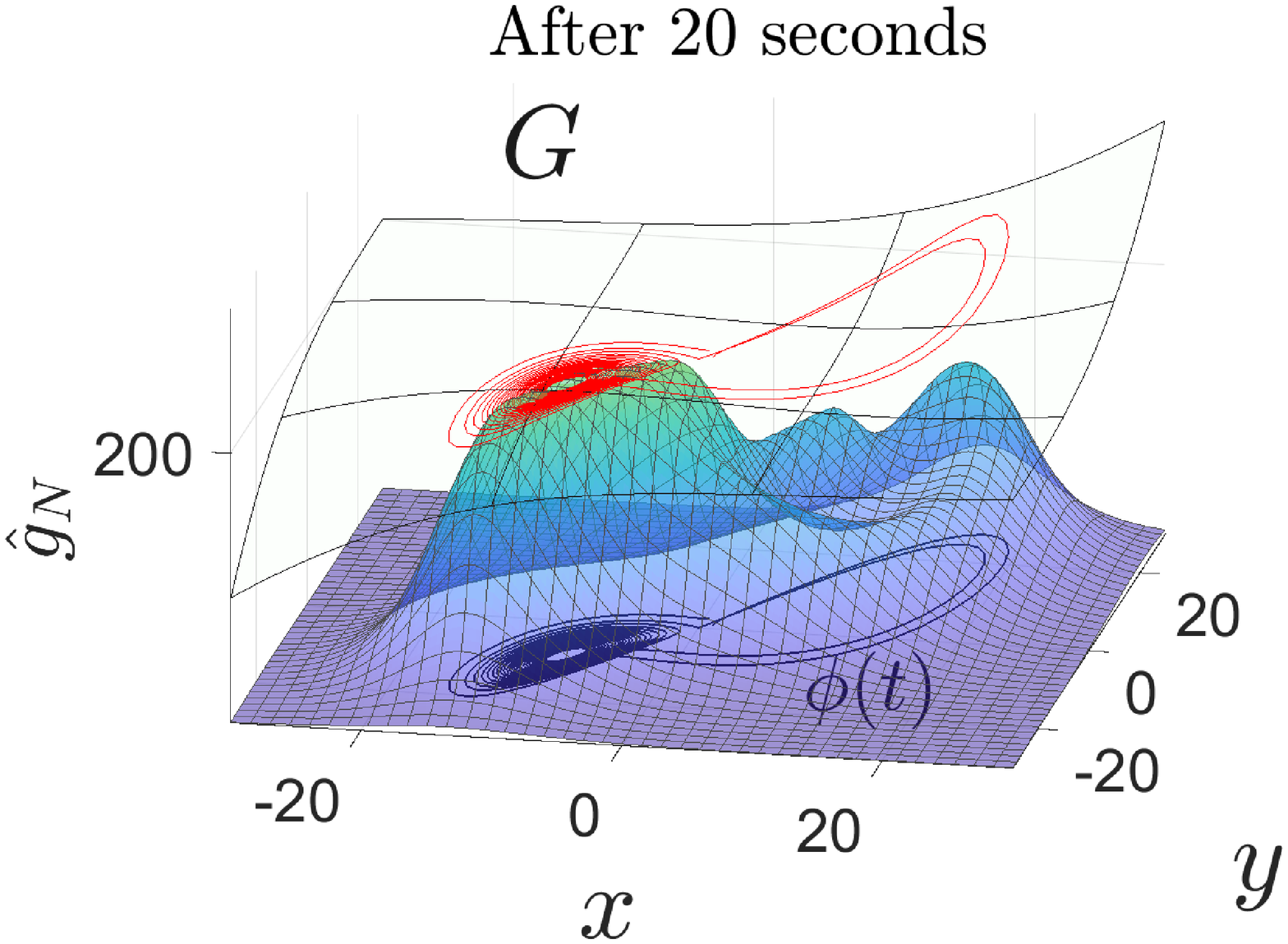}
    \caption{}
    \label{fig:estimate20s-1}
    \end{subfigure}
    \begin{subfigure}[b]{0.47\textwidth}
    \includegraphics[width=\textwidth]{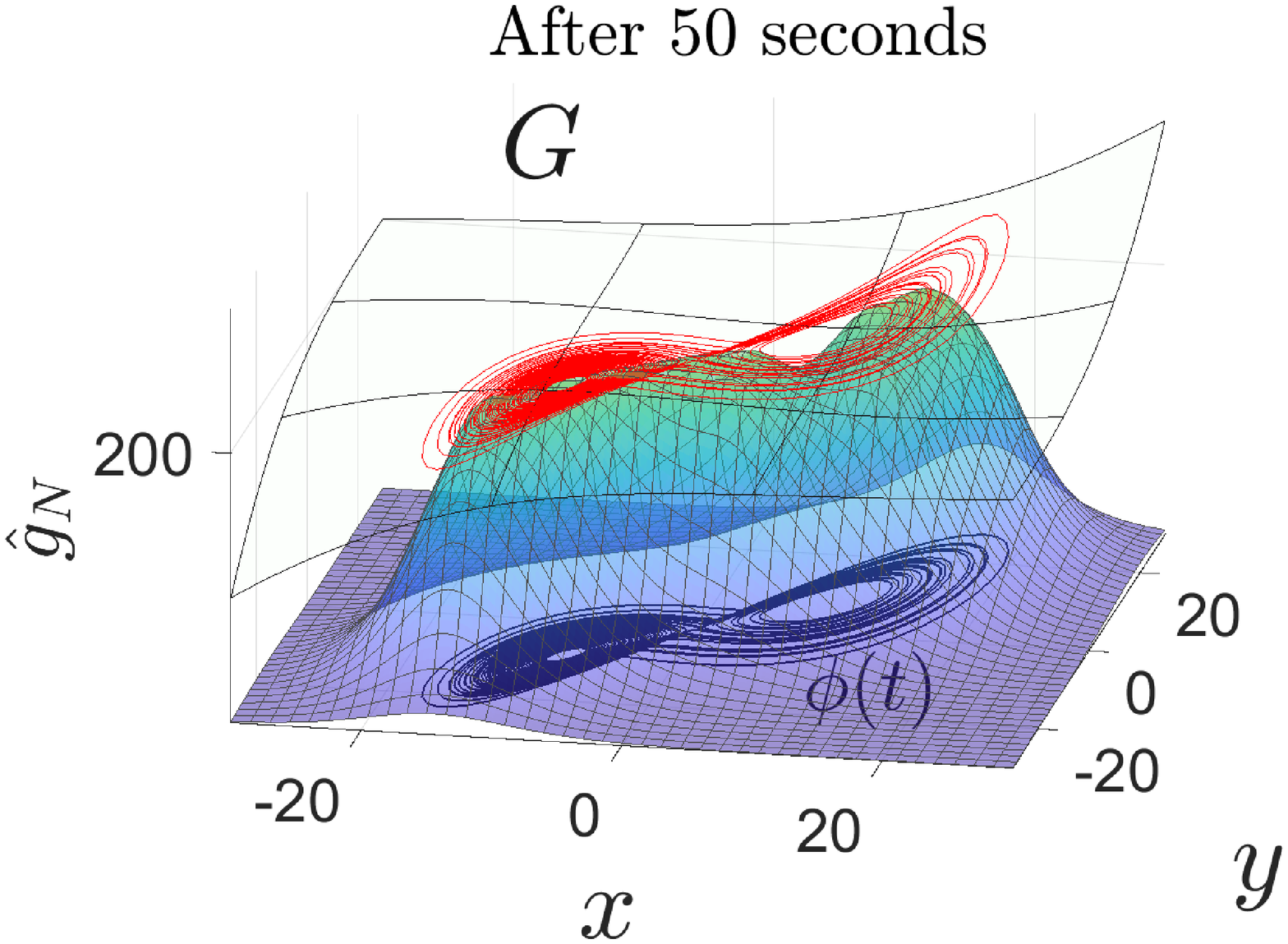}
    \caption{}
    \label{fig:estimate50s-1}
    \end{subfigure}
     \hfill
    \begin{subfigure}[b]{0.47\textwidth}
    \includegraphics[width=\textwidth]{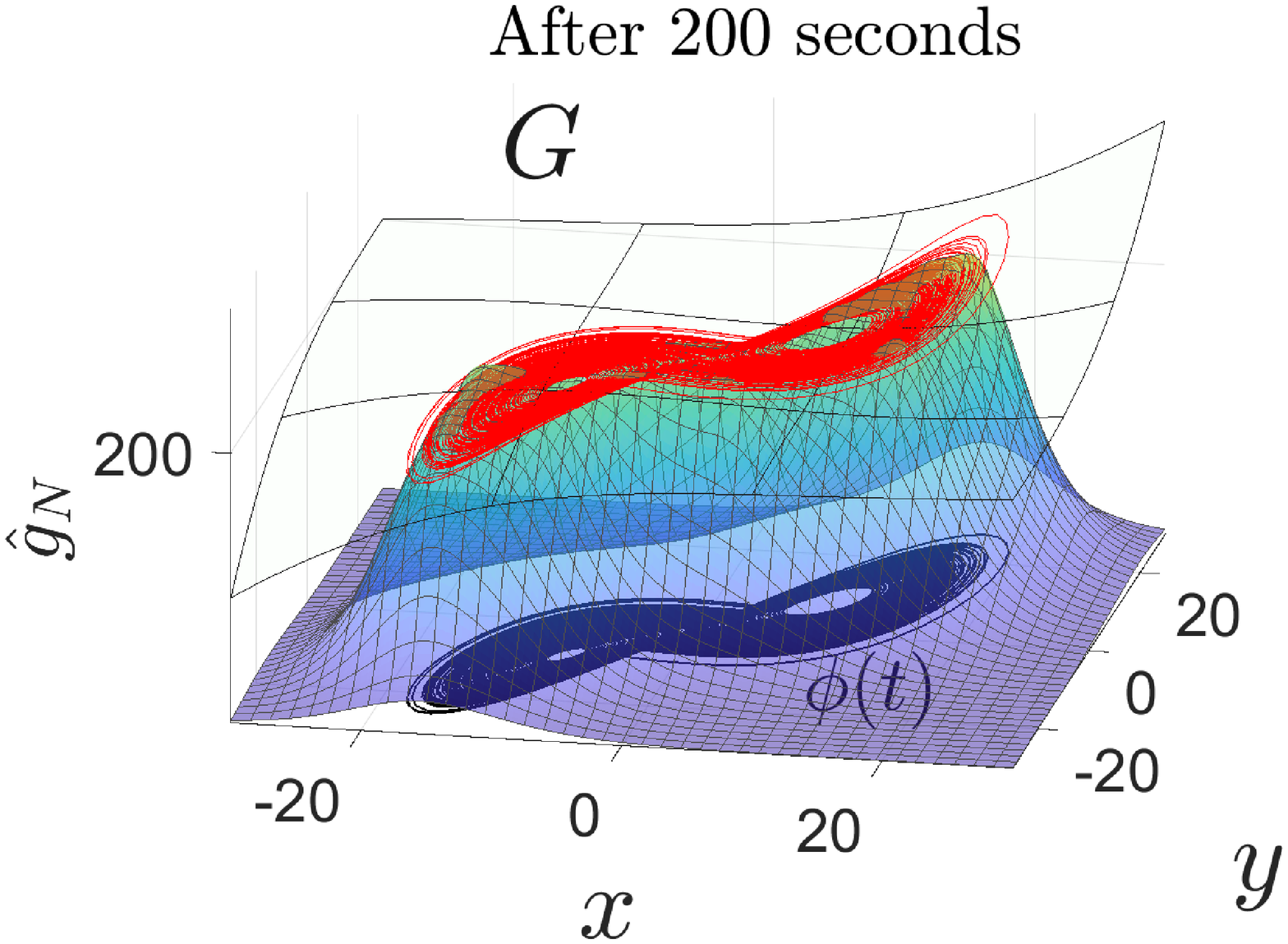}
    \caption{}
    \label{fig:estimate200s-1}
    \end{subfigure}
    \caption{An illustration of the various approximations from Equation \ref{eq:approxN1} over an orbit after spanning different amounts of time from the same system dynamics and initial conditions. It is evident that the error $\|\Pi_{N} G-\hat{g}_N(t,\cdot)\|$ converges to zero over the domain of attraction as more time is passed and more samples are collected.}
\end{figure}

For the next set of figures, we examine the effects of the regularization penalty term of $E(t,g;\phi)$ used in approximation Method (1) by varying the choices of $\gamma$. In order to examine the influence of the regularization parameter on the long-term convergence behavior, each simulation is run over a sufficiently long time interval of 200 seconds. Using the same Lorenz system, kernel function, hyperparameter $\beta$, unknown function $G$, initial conditions, and center spacing as the previous results, Figures \ref{fig:estimateGamma10-1} through \ref{fig:estimateGammaP001-1} illustrate the estimates for different values of $\gamma$. Overall, these four graphs depict qualitative behavior that is well-known in the field of inverse problems. The error functional $E$ introduced in Section \ref{sec:offlineoptimal} balances two terms 
\begin{equation*}
    E(t,g;\phi) := \underbrace{\frac{1}{2}\int|y(\tau)-\EE_{\phi(\tau)}g|^2 \nu(d\tau)}_{\text{term 1}} + \underbrace{\frac{1}{2}\gamma \|g \|_{\sH}^2}_{\text{term 2}}.
\end{equation*}
Minimizing term 1 decreases the error over the orbit $\Gamma(\phi_0)$, while term 2 penalizes the size of the estimate as measured in the $\| \cdot \|_{\sH}$ norm. Increasing $\gamma$ generally leads to smoother estimates. Additionally, it also reduces the chance of over-fitting, which can lead to poor estimates in the presence of a noisy or perturbed data set. This classic phenomenon is studied in great depth in texts like \cite{engl1996regularization}. With these considerations in mind, proper estimates can be generated by selecting a particular choice of $\gamma$ to effectively regularize the estimate.
\begin{figure}[ht!]
    \centering
    \begin{subfigure}[b]{0.47\textwidth}
    \includegraphics[width=\textwidth]{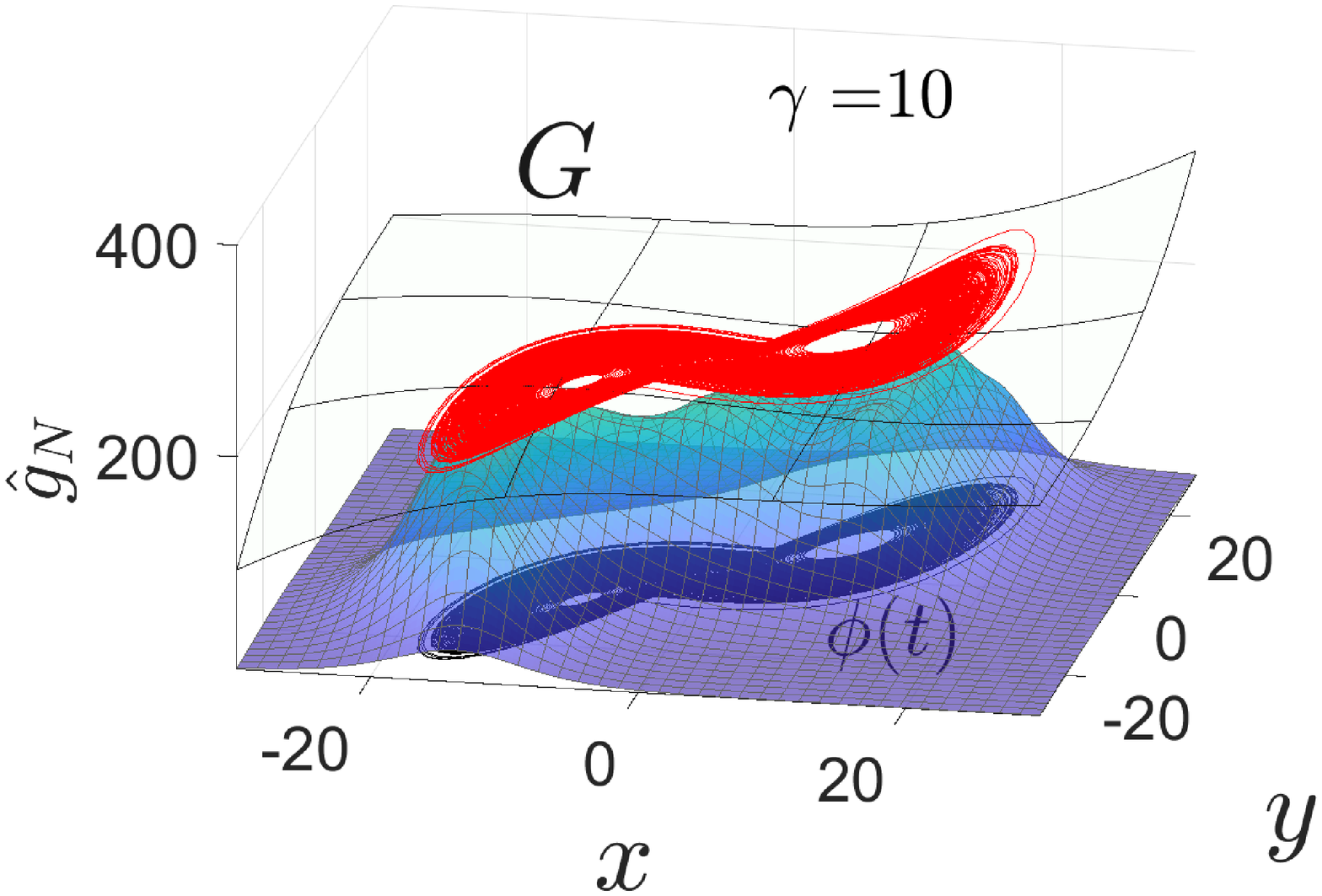}
    \caption{}
    \label{fig:estimateGamma10-1}
    \end{subfigure}
    \begin{subfigure}[b]{0.47\textwidth}
    \includegraphics[width=\textwidth]{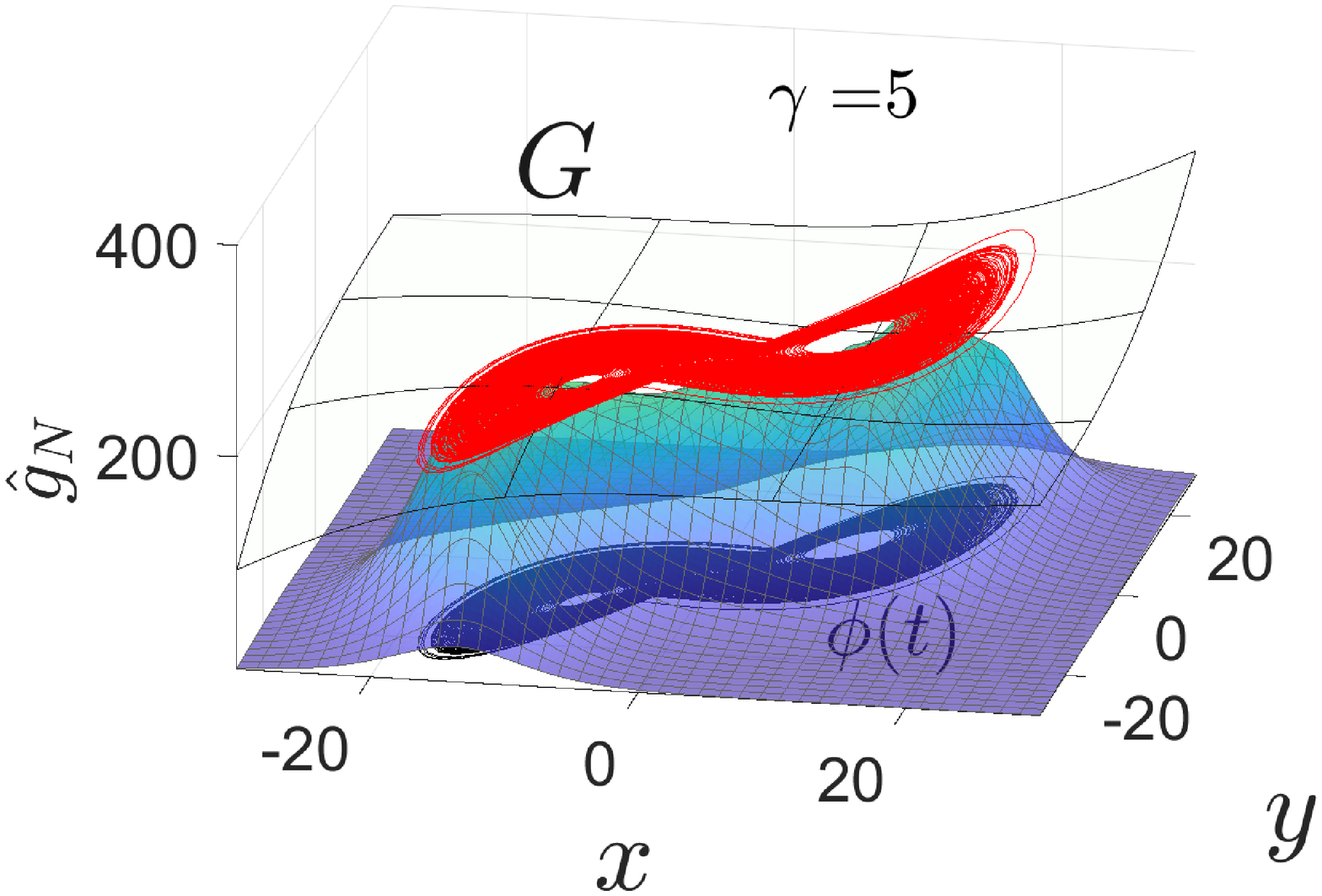}
    \caption{}
    \label{fig:estimateGamma5-1}
    \end{subfigure}
    \hfill
    \begin{subfigure}[b]{0.47\textwidth}
    \includegraphics[width=\textwidth]{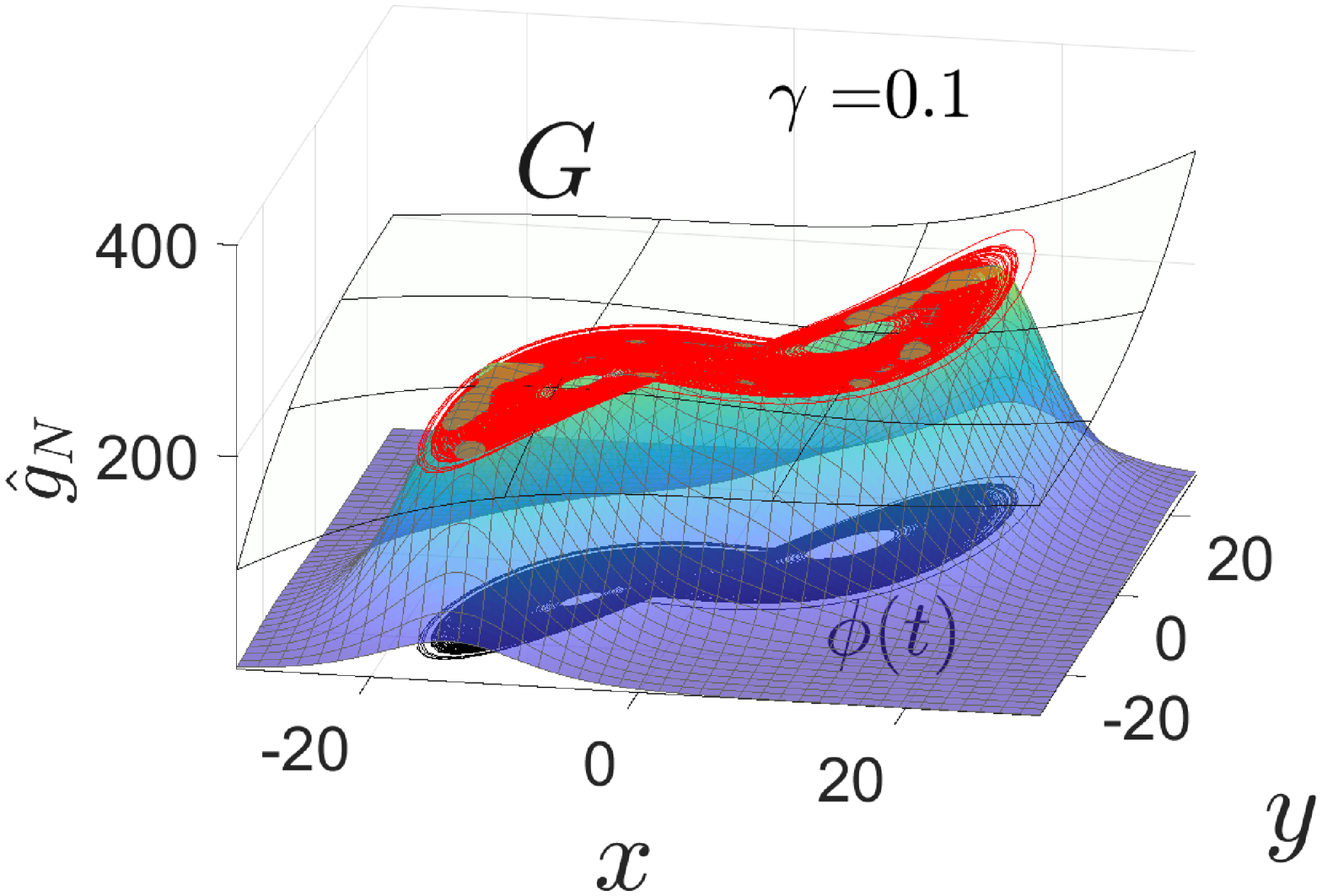}
    \caption{}
    \label{fig:estimateGammaP1-1}
    \end{subfigure}
    \begin{subfigure}[b]{0.47\textwidth}
    \includegraphics[width=\textwidth]{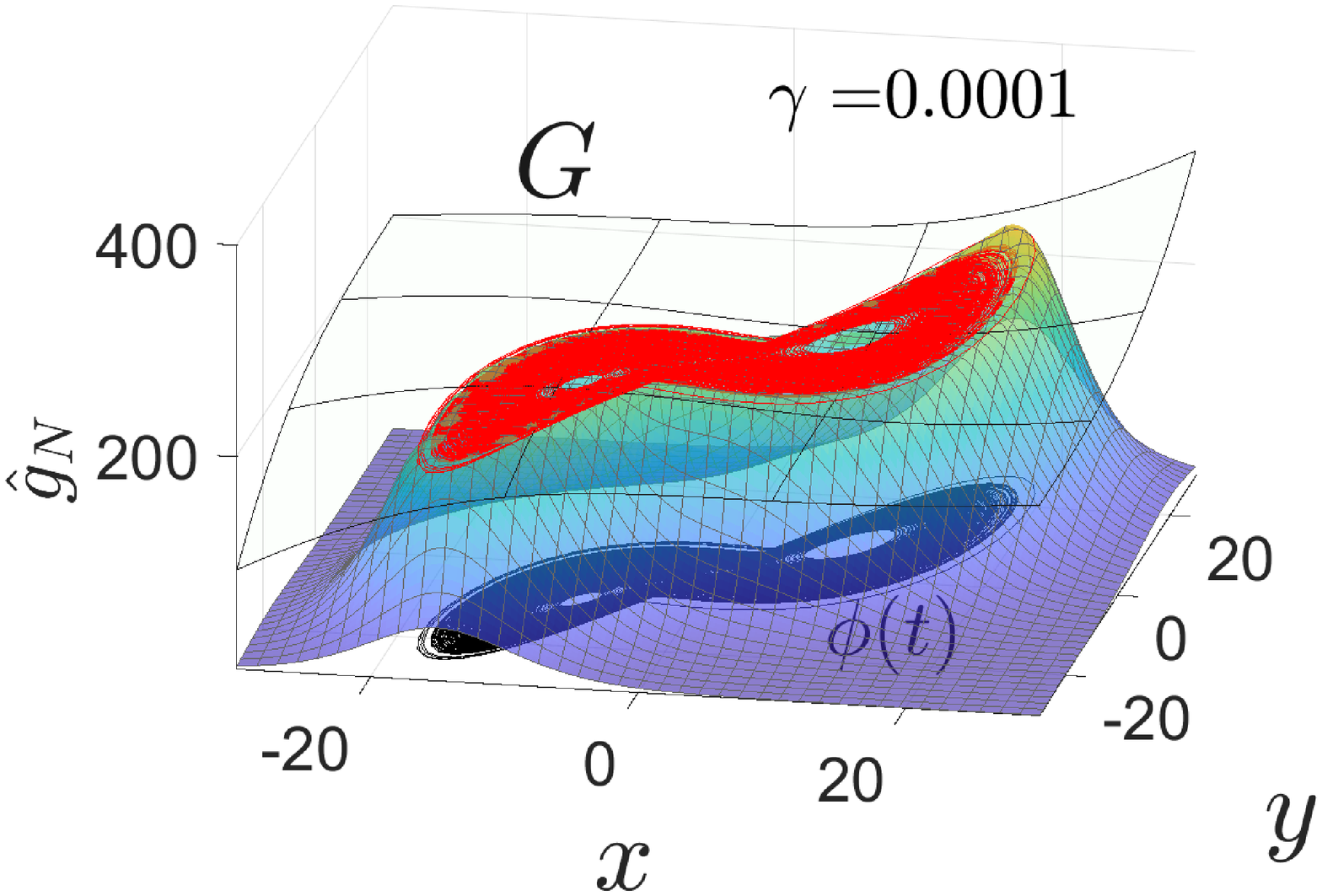}
    \caption{}
    \label{fig:estimateGammaP001-1}
    \end{subfigure}
    \caption{Estimates $\hat{g}_N(t,\cdot)$ of the function $G$ for different values of $\gamma$ using the same time-span, initial condition, and underlying dynamics. Notice that for larger $\gamma = 20$, the estimate has a slower convergence over the domain of attraction. While estimates using smaller $\gamma$, may have reduced error over the orbit, the estimate can potentially over-fit noisy training data.}
\end{figure}

\subsection{Example: Characteristics of Approximation Method (2)}
The next set of results examine how the estimates generated using approximation Method (2) vary in time and converge as $t \to \infty$. Using the same conditions as the previous results, the estimate was initialized with all coefficients $\alpha_i = 0$ for $1 \leq i \leq N$ and the states initial condition $\phi_0 = \{1,1,1\}^T$. Figures \ref{fig:estimate10s-2} through \ref{fig:estimate200s-2} illustrate the estimate generated after different amounts of time. In contrast to the previous approximation method, these figures depict the evolution of a single estimate as $t \to \infty$.  As mentioned previously, this approximation is determined by an evolution of coefficients $\{\hat{\alpha}_{N,i}(t)\}_{i=1}^N$. The evolution law minimize the integrated error for $\hat{g}_N(\tau) =\sum_{i=1}^N \hat{\alpha}_{N,i}(\tau)\knl_{\xi_{N,i}}(\cdot)$ over the orbit for time $\tau \in [0,t]$ rather than computed an offline optimal solution at a specific time.  At early stages of the evolution such as Figure \ref{fig:estimate10s-2}, samples predominantly aggregate near one of the unstable equilibrium points. However, as the trajectory approaches the second equilibrium, it begins to influence the coefficients of the nearby centers and decrease the error between the estimate $\hat{g}_N(t,\cdot)$ and $G$. Figures \ref{fig:estimate50s-2} and \ref{fig:estimate200s-2} suggest that the estimate converges to the projection $\Pi_N G$ over the domain of attraction as more time passes and more samples are collected. Similar to the estimate from approximation Method (1), there is an initial rapid change in the error over the trajectory. While the error appears to decrease for longer periods of time, it is evident that  the error reduction between Figures \ref{fig:estimate50s-2} and \ref{fig:estimate200s-2} occurs at a much lower rate than the reduction seen from Figure \ref{fig:estimate10s-2} to \ref{fig:estimate20s-2} and Figure \ref{fig:estimate20s-2} to \ref{fig:estimate50s-2}.

\begin{figure}[ht!]
    \centering
    \begin{subfigure}[b]{0.47\textwidth}
    \includegraphics[width=\textwidth]{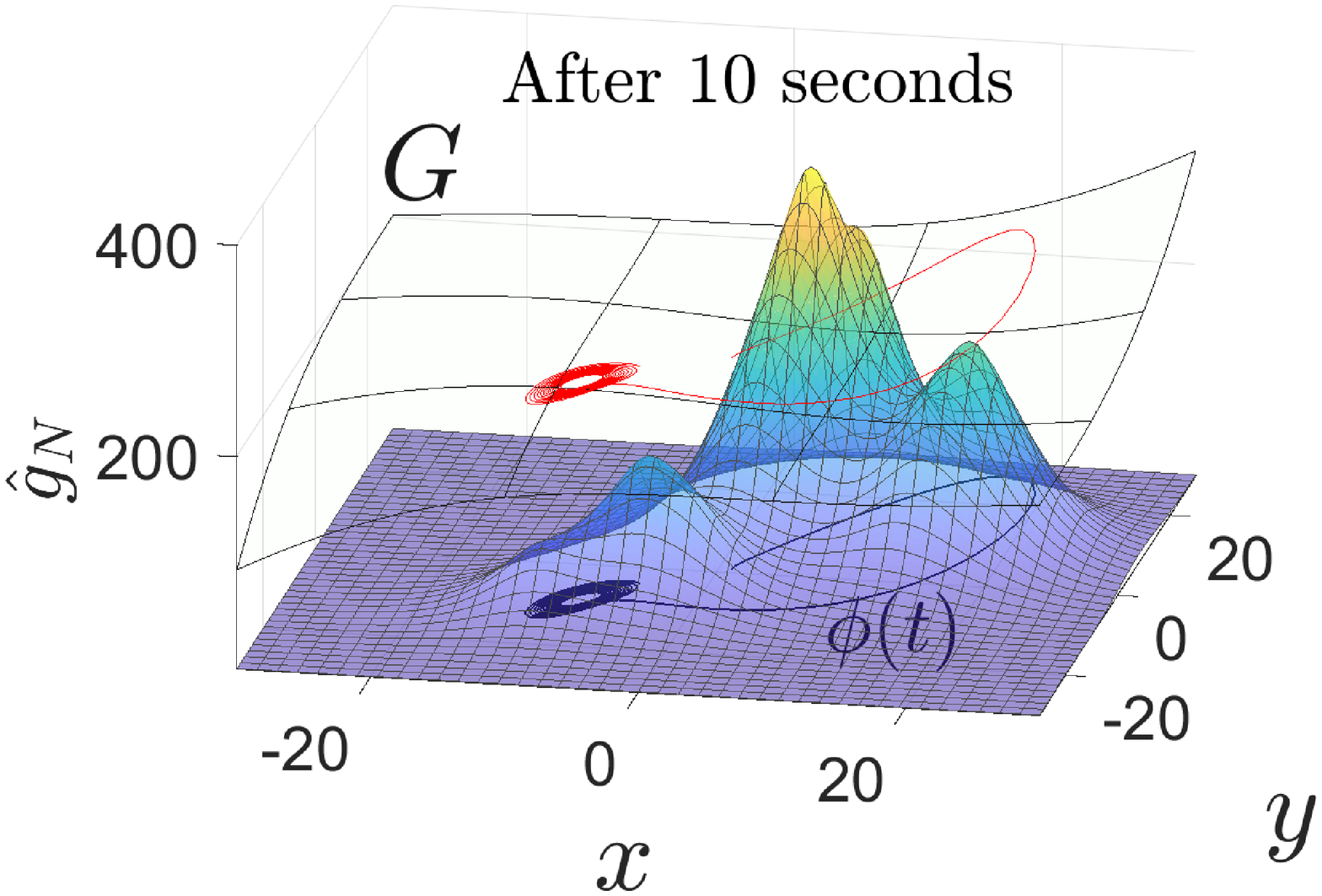}
    \caption{}
    \label{fig:estimate10s-2}
    \end{subfigure}
    \hfill
    \begin{subfigure}[b]{0.47\textwidth}
    \includegraphics[width=\textwidth]{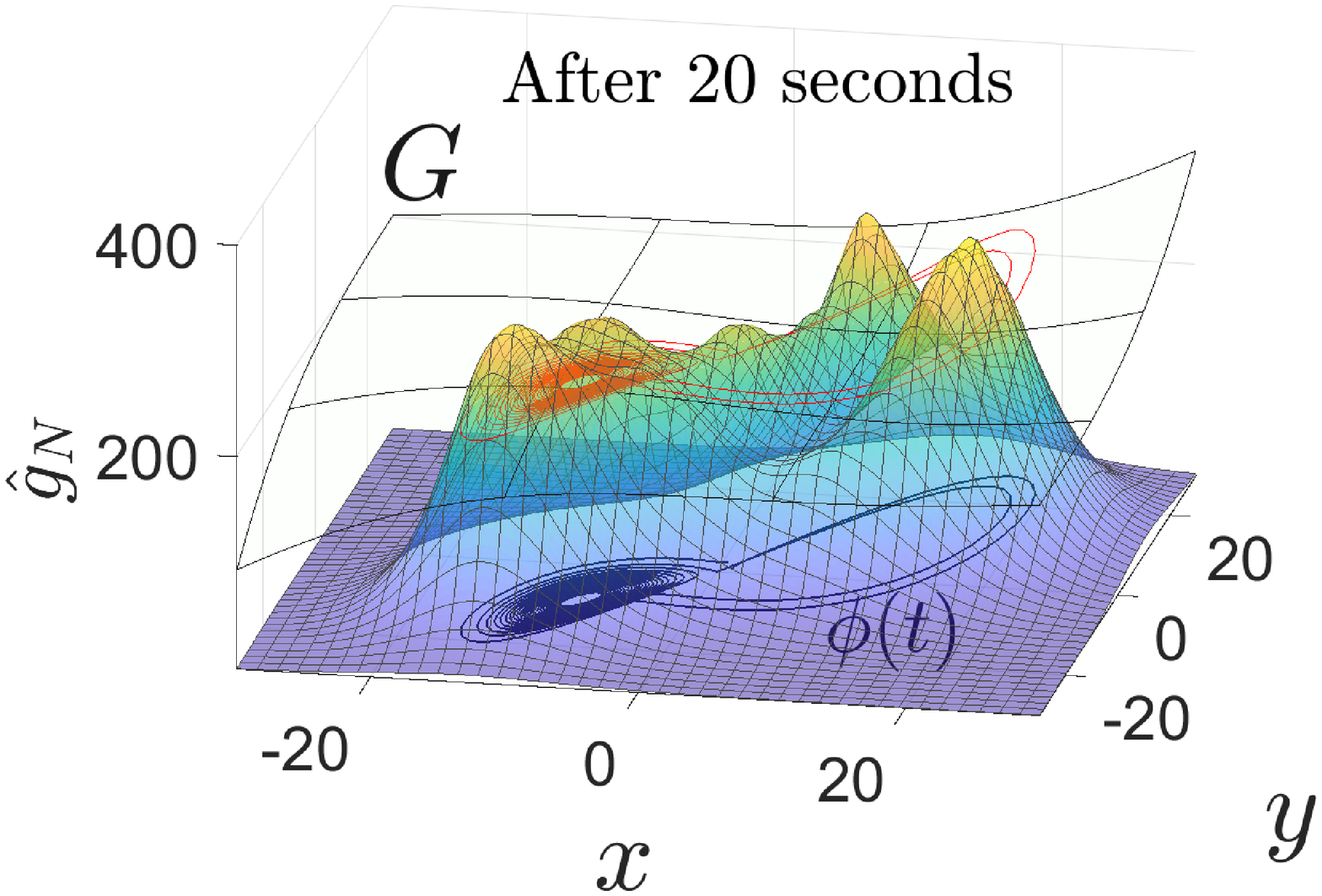}
    \caption{}
    \label{fig:estimate20s-2}
    \end{subfigure}
    \begin{subfigure}[b]{0.47\textwidth}
    \includegraphics[width=\textwidth]{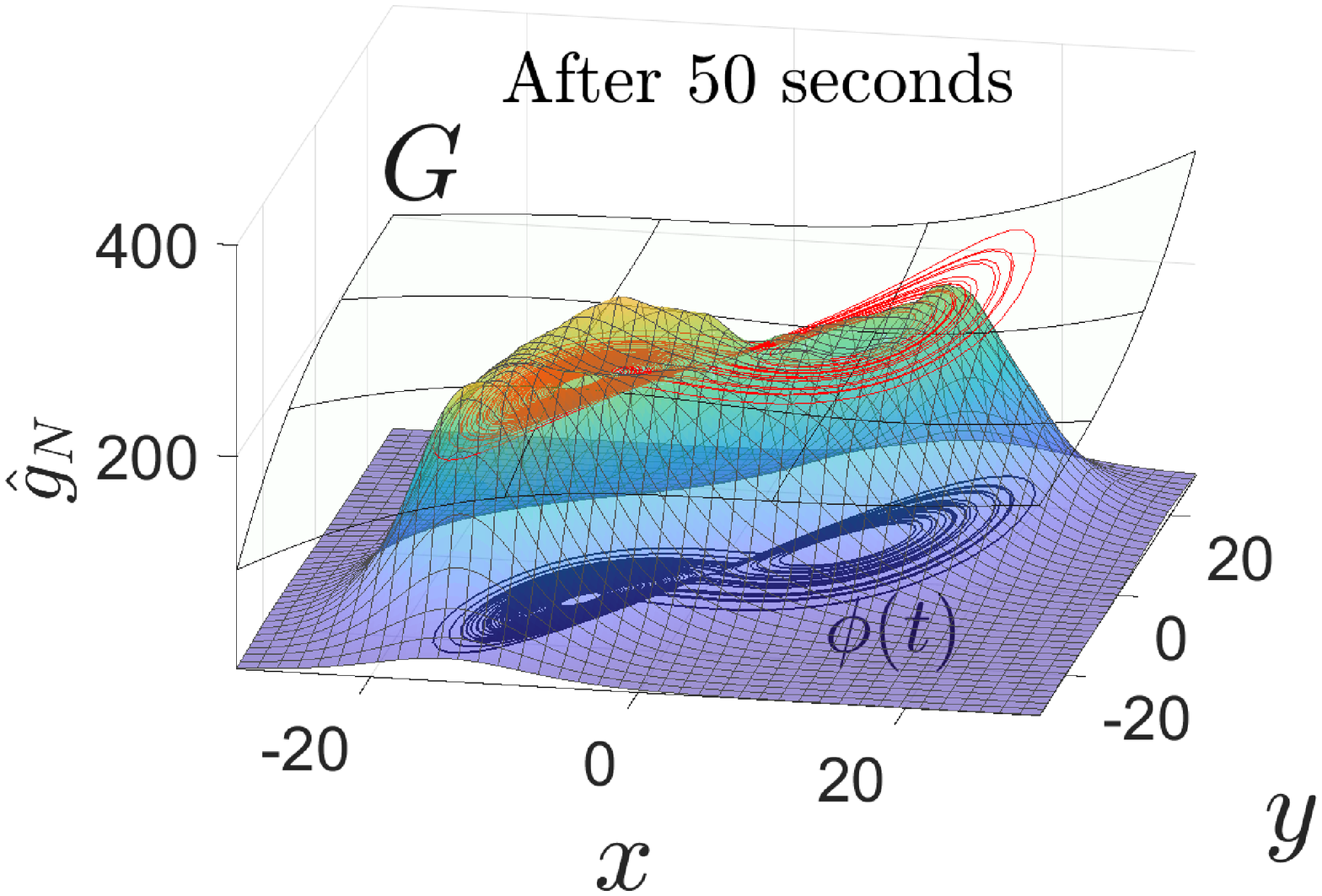}
    \caption{}
    \label{fig:estimate50s-2}
    \end{subfigure}
     \hfill
    \begin{subfigure}[b]{0.47\textwidth}
    \includegraphics[width=\textwidth]{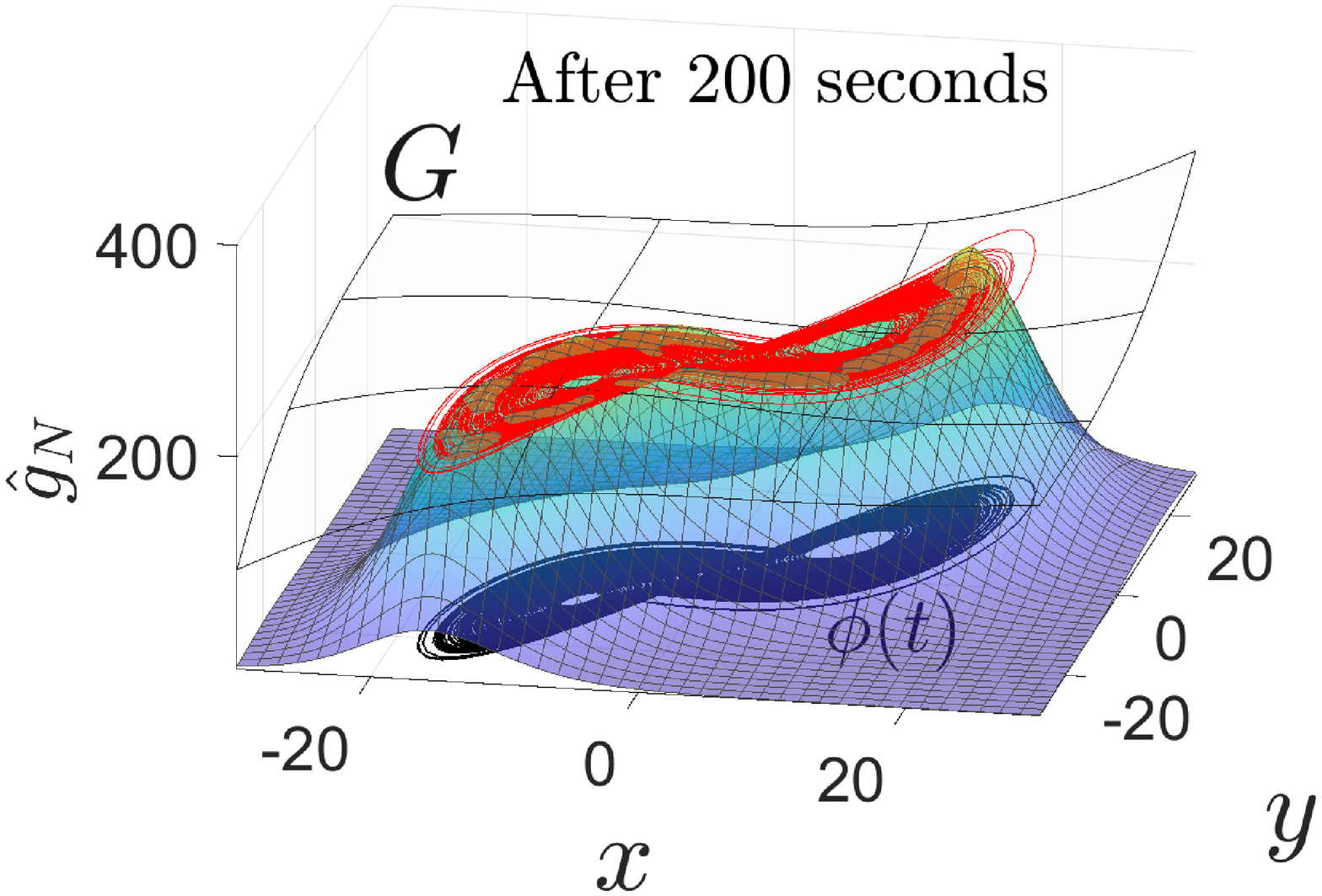}
    \caption{}
    \label{fig:estimate200s-2}
    \end{subfigure}
    \caption{An illustration of the evolution approximation from Equation \ref{eq:approxN2} over an orbit as time $t \to \infty$. The error  $\|\Pi_{N} G-\hat{g}_N(t,\cdot)\|$ converges over the domain of attraction as more time is passed and more samples are collected.}
\end{figure}

Using the same conditions as those used in the regularization study of approximation Method (1), we also examine the effects of the regularization term on the estimate from approximation Method (2) by building estimates using different choices of $\gamma$. Figures \ref{fig:estimateGamma10-2} through \ref{fig:estimateGammaP0001-2} illustrate the estimates for different values of $\gamma$.  In these estimates, the $\gamma$ term also plays a role in the transient response of the coefficients' evolution. From the figures, it is evident that a larger $\gamma$ decreases the sensitivity to changes in the estimate as new samples are collected over time. Consequently, the numerical study suggests better convergence of the estimates may require a larger number of samples when $\gamma$ is large. 
  By decreasing $\gamma$, the estimates are more responsive to changes in the data as seen in Figure \ref{fig:estimateGammaP1-2}. However, even without noise or disturbances, smaller values of $\gamma$ can yield large oscillatory behavior in the estimate as seen in Figure \ref{fig:estimateGammaP0001-2}. This example indicates that $\gamma$ must be carefully selected for a desired transient response in the second approximation method.
\begin{figure}[ht!]
    \centering
    \begin{subfigure}[b]{0.47\textwidth}
    \includegraphics[width=\textwidth]{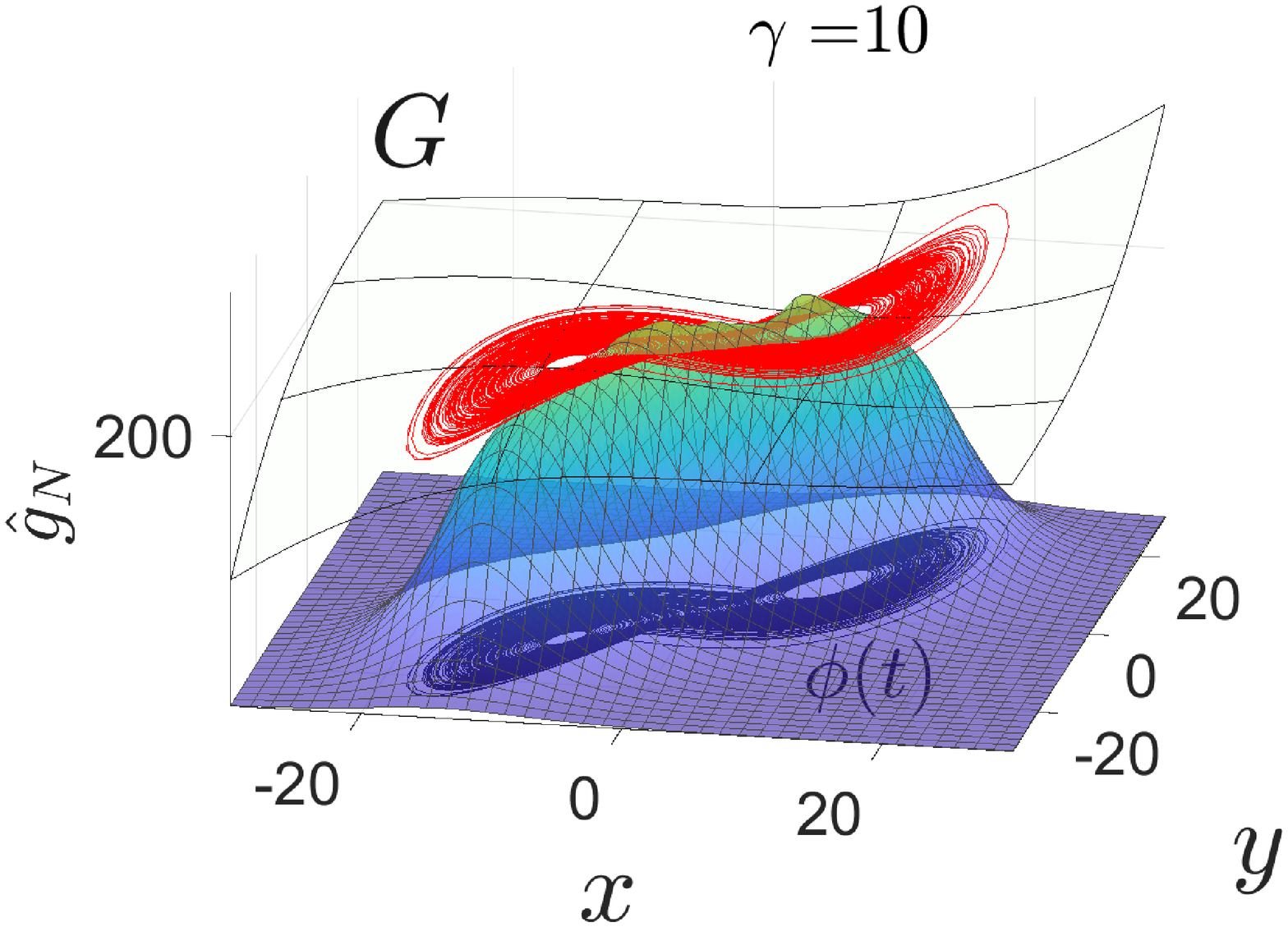}
    \caption{}
    \label{fig:estimateGamma10-2}
    \end{subfigure}
    \begin{subfigure}[b]{0.47\textwidth}
    \includegraphics[width=\textwidth]{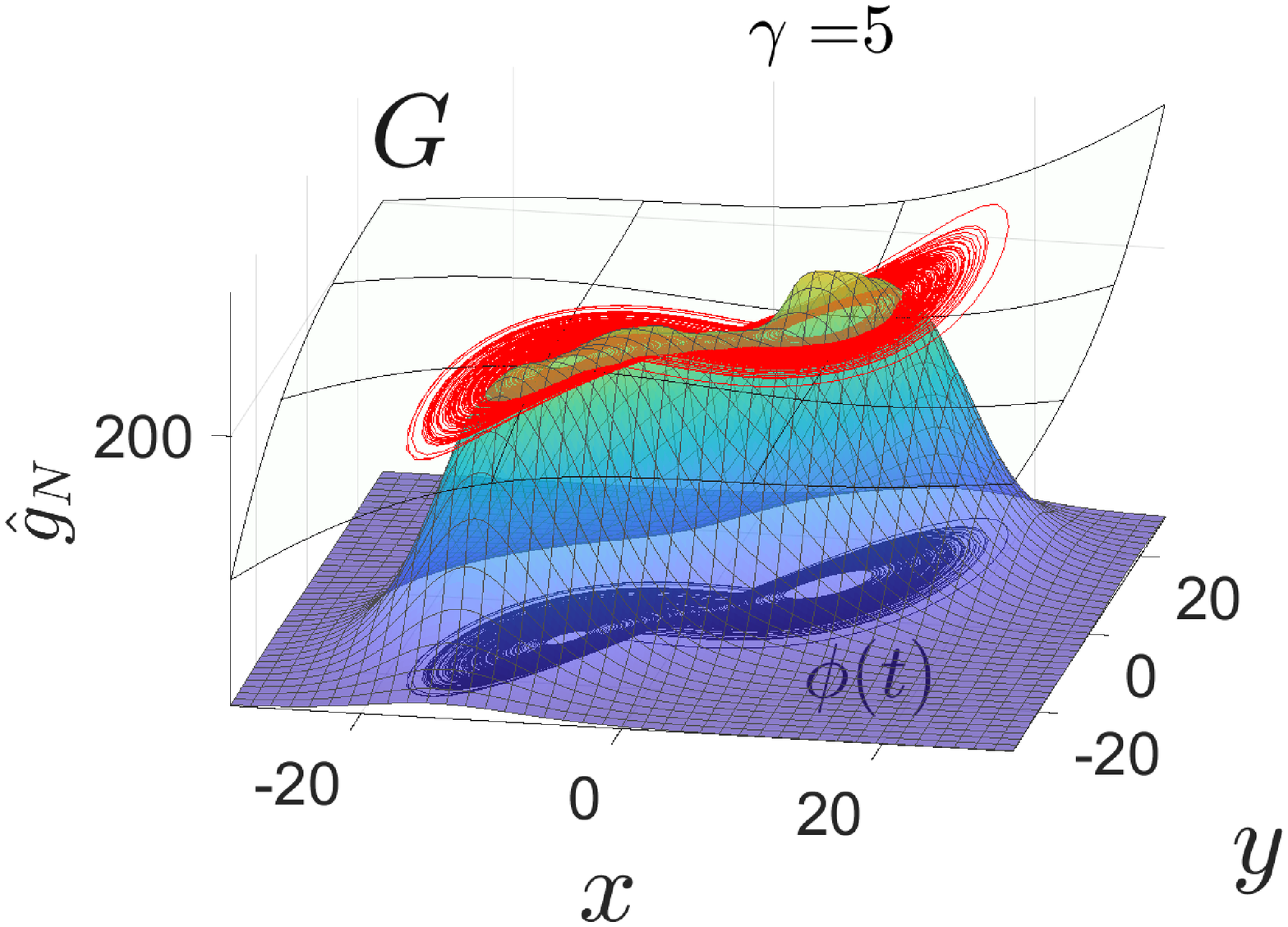}
    \caption{}
    \label{fig:estimateGamma5-2}
    \end{subfigure}
    \hfill
    \begin{subfigure}[b]{0.47\textwidth}
    \includegraphics[width=\textwidth]{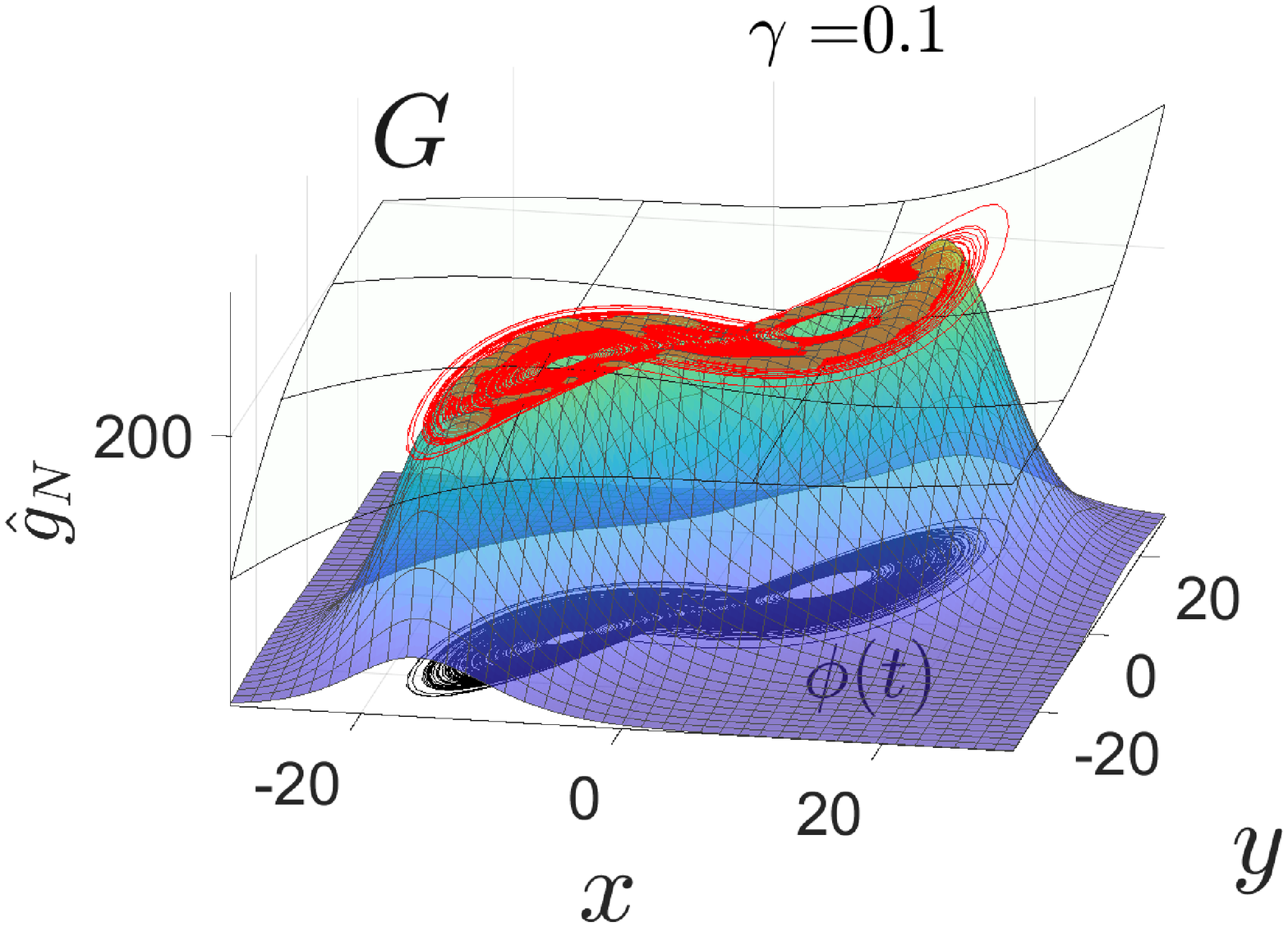}
    \caption{}
    \label{fig:estimateGammaP1-2}
    \end{subfigure}
    \begin{subfigure}[b]{0.47\textwidth}
    \includegraphics[width=\textwidth]{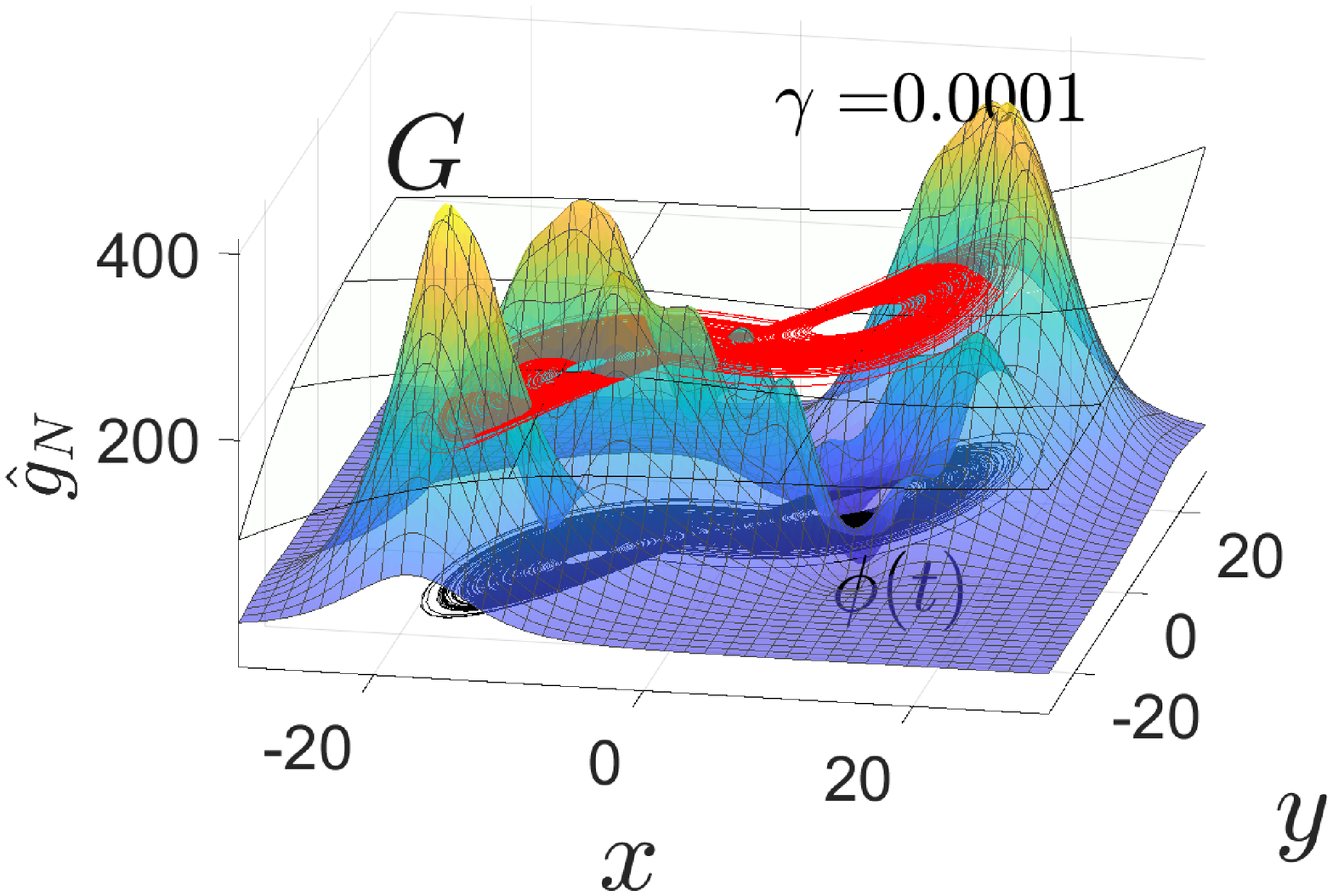}
    \caption{}
    \label{fig:estimateGammaP0001-2}
    \end{subfigure}
    \caption{Demonstration of estimates $\hat{g}_N(t,\cdot)$ of the function $G$ for different values of $\gamma$ using the same time-span, initial condition, and underlying dynamics. Notice that the larger $\gamma = 20$ has a slower convergence over the domain of attraction. For smaller $\gamma$, such the estimate generated using $\gamma = 0.001$, the estimate suffers from less penalty on the regularization of the estimate. However, the estimate can over-fit noisy training data.}
\end{figure}
\subsection{Example: Human Kinematics Study}
This example uses three-dimensional motion capture data from a subject running along a treadmill \cite{fukuchi2018public}. From the experiment, 17,000 marker coordinates are collected relative to a fixed inertial frame defined by the camera's position. For this example, a small candidate kinematic model is defined in terms of the full collection of experimental trajectories. The marker coordinates of the hip, knee, and ankle in the full data set are projected onto the sagittal plane that divides the left and right half of the body, see Figure \ref{fig:kinematicFigure}. 

With this projection, the first input is defined to be the joint angle  $\theta^{(1)}$. It is measured between the projected vector $v^{(1)}$ that connects the hip to knee and the body-fixed $b^{(1)}$ vector in the plane, and it roughly corresponds to hip flexion. The second input comes from the knee flexion angle $\theta^{(2)}$. It is measured between the vector $v^{(1)}$ and $v^{(2)}$, the vector connecting the knee to the ankle projected to the sagittal plane. These choices of projections and associated degrees of freedom are taken to define a low-dimensional but unknown dynamic system. We seek to estimate outputs of the unknown dynamic system.

For illustrative purposes, we chose to estimate the projection $\Pi_N G$ of the kinematic map $G$ from  the joint angles $\theta^{(1)}$ and $\theta^{(2)}$ to the ankle coordinate associated with the body-fixed $b^{(2)}$ vector. Figure \ref{fig:kernelMocapEstimate1} and \ref{fig:kernelMocapEstimate2} show the approximations $\hat{g}_{N}(t,\cdot)$ of $\Pi_N G$  over an orbit in the input space defined by the coordinates $\theta^{(1)}$ and $\theta^{(2)}$. The estimates are generated for both approximation Method (1) and (2), respectively. In both estimates, the Matern-Sobolev kernel with $\beta =10 $ is used. Additionally, the kernel centers for both estimates are chosen so that there is sufficient separation distance of at least 10 radians between each of the kernel centers, $\Xi_N$. As mentioned previously, the two estimates respond differently to the regularization term. A relatively small regularization parameter $\gamma = 0.001$ was selected for the estimate from approximation Method (1). The second approximation exhibited oscillations over the data for small values of $\gamma$. Consequently, we increased the regularization parameter to $\gamma = 2$ for the estimate from the second approximation method. 

When comparing these approximations there are a couple of key things to note. The estimate using the second approximation method in this study is generated using a continuous evolution law  as given by Equation \ref{eq:approxNevolution}. However, the collected data consists of discrete sets of motion capture data collected at discrete times. In order to utilize the collected data and also have a continuous evolution update, we use MATLAB's built-in functions to fit splines that are continuous in time with knots at the discrete states. This builds a continuous approximation of the orbit over small time intervals. While the approximation given by Method (1) does not need this step, we must approximate the integral over the orbits given the discrete data and a particular quadrature rule. Consequently, the estimate of Method (1) is susceptible to error associated with quadrature approximations. 

\begin{figure}[htp]
    \centering
     \includegraphics[width = 0.9 \textwidth]{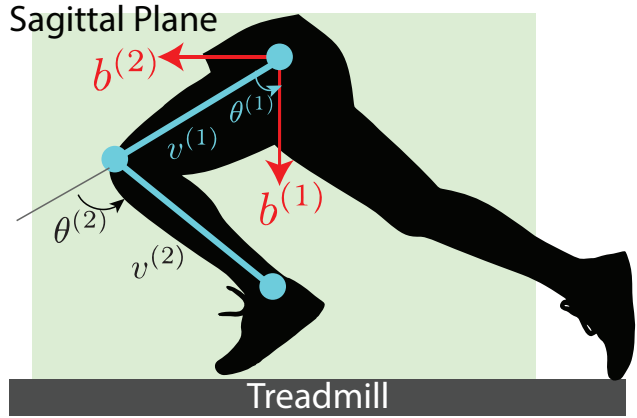}
    \caption{An illustration of the inputs $\theta^{(1)}$ and $\theta^{(2)}$, which roughly correspond to hip and knee flexion respectively.  These input variables can be mapped to measured marker coordinates placed on joints such as the knee or ankle.}
    \label{fig:kinematicFigure}
\end{figure}

\begin{figure}[htp]
    \centering\
    \begin{subfigure}[b]{0.47\textwidth}
        \includegraphics[width = \textwidth]{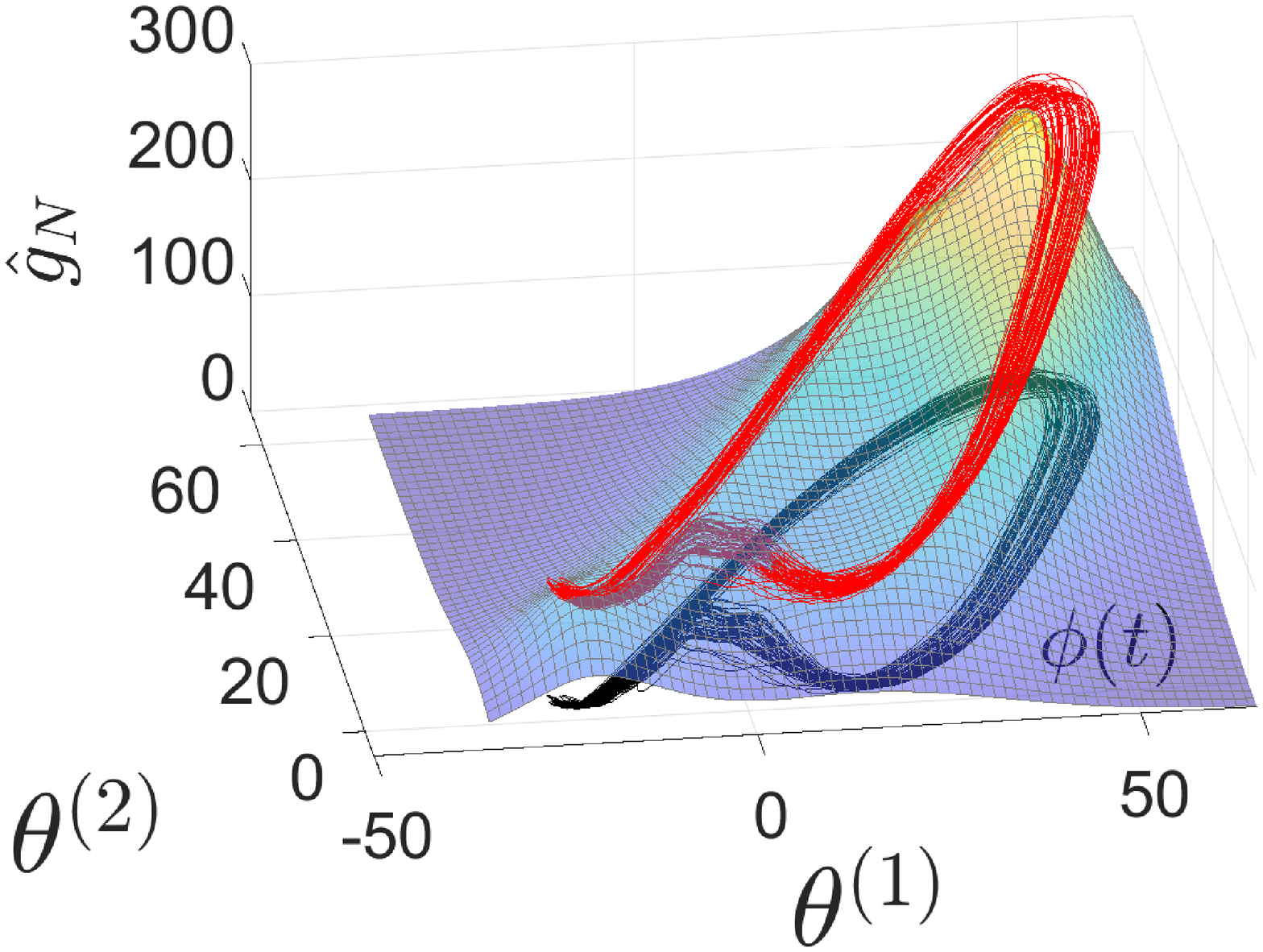}
         \caption{}
        \label{fig:kernelMocapEstimate1}
    \end{subfigure}
    \begin{subfigure}[b]{0.47\textwidth}
        \includegraphics[width=\textwidth]{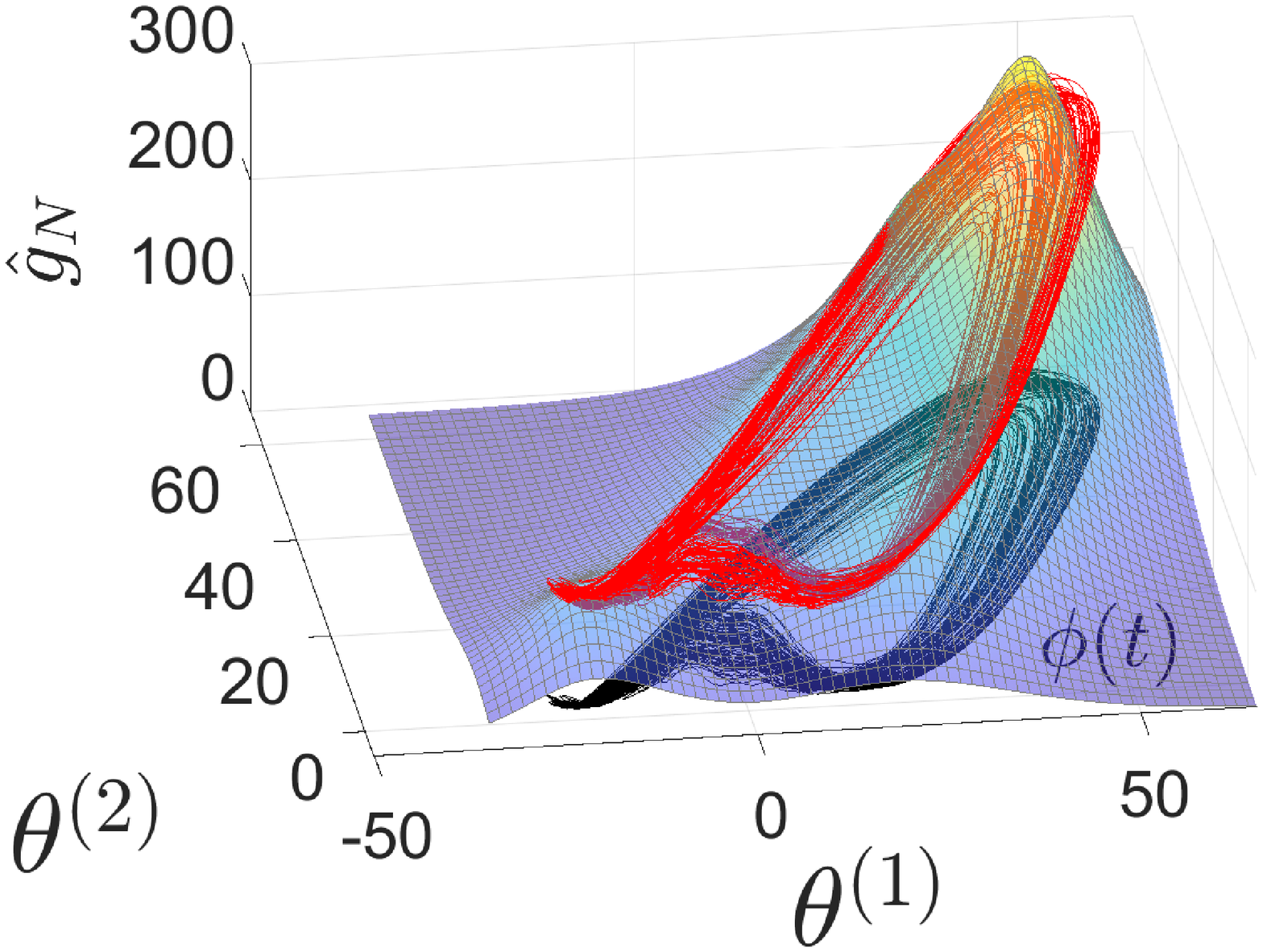}
        \caption{}
        \label{fig:kernelMocapEstimate2}
        \end{subfigure}
    \caption{ The estimates $\hat{g}_N(t,\cdot)$ of the ankle coordinate in the $b^{(2)}$ direction over the input space given by the angles $\theta^{(1)}$ and $\theta^{(2)}$ associated with hip and knee flexion respectively using (a) Approximation Method (1), and (b) Approximation Method (2). It must be noted that the first Approximation method shows relatively small error for the given sample size compared to Method (2). }
\end{figure}

\section{Conclusions}
In this paper, an optimal, offline estimation problem is formulated for continuous time regression over state spaces that include certain types of smooth manifolds. A new persistency of excitation condition is introduced that is well-defined for flows on manifolds, and it is used to obtain sufficient conditions for convergence.  Error estimates are derived that characterize the rate of convergence of finite dimensional estimates of the solution of the regression problem in continuous time over manifolds.  We then discuss two methods to generate finite-dimensional approximations of the optimal regression estimate. Numerical simulations are presented to better illustrate the qualitative behavior of the two approximation methods. Finally, we discuss and analyze results on estimating functions over motion capture data to demonstrate how to implement the algorithm on experimental data.

\section{Appendix}
\subsection{Background on Galerkin Approximations} 
\label{sec:galerkinbackground}
Let $U$ be a real Hilbert space, $A\in \calL(U)$ be a  bounded linear operator on $U$,    $b\in U$ be a fixed element of $U$, and suppose we seek to find $u\in U$ that satisfies the operator equation 
\begin{equation*}
Au=b.
\end{equation*}
It is customary that the existence and uniqueness of the solution of this equation is established by studying the associated bilinear form $a(\cdot,\cdot):U\times U \rightarrow \RR$ given by $a(u,v):=(Au,v)$ for all $u,v\in U$. Then the operator equation above is equivalent to finding the $u\in U$ for which 
\begin{equation}
    a(u,v)=\langle b,v\rangle_U \quad \text{ for all } v\in U. \label{eq:weakeq2}
\end{equation}
The Lax-Milgram Theorem given below stipulates a concise pair of conditions that ensure the well-posedness of the operator equation. 

\begin{theorem}[Lax-Milgram Theorem \cite{ciarlet2013linear}]
\label{th:laxmilgram}
The bilinear form $a(\cdot,\cdot):U\times U \rightarrow \RR$ is bounded if there is a constant $C_1>0$ such that 
\begin{align*}
    |a(u,v)|\leq C_1 \|u\|_U \|v\|_U \quad \text{ for all } u,v\in U, 
\end{align*}
and it is coercive if there is a constant $C_2>0$ such that 
\begin{equation*}
    |a(u,u)|\geq C_2 \|u\|^2_U \quad \text{ for all } u\in U. 
\end{equation*}
If the bilinear form $a(\cdot,\cdot)$ is bounded and coercive, then $A^{-1}\in \calL(U)$ and there is a unique solution $u\in U$ to Equation \ref{eq:weakeq2}.
\end{theorem}
\begin{proof}
The first condition above, continuity of the bilinear form,  ensures that $A\in \calL(U)$ by definition. The coercivity condition implies that the nullspace of $A$ is just  $\{0\}$. As a result, we know that $A$ is one-to-one. This means that the operator $A^{-1}:\text{range}(A)\rightarrow U$ is well-defined. From the coercivity condition we also conclude that 
\begin{equation*}
    \|A^{-1}b\|^2 \leq \frac{1}{C_2}|\langle A\cdot A^{-1}b,A^{-1}b\rangle_U|\leq \frac{1}{C_2} \|b\|_U \|A^{-1}b\|_U
\end{equation*}
for every $b\in \text{range}(A)$. This means that $A^{-1}\in \calL(\text{range}(A),U)$ and $\|A^{-1}\|\leq 1/C_2$. 

One implication of the fact that $A^{-1}\in \calL(\text{range}(A),U)$  is that $\text{range}(A)$ is closed.  Suppose that $\{b_k\}_{k\in \NN}\subset \text{range}(A)$ and $b_k\rightarrow \bar{b}$. By construction there is a sequence $\{u_k\}_{k\in \NN}\subset U$ such that $Au_k=b_k$.   But we have 
\begin{equation*}
    \|u_m-u_n\|= \|A^{-1}(y_m-y_n)\|_U \leq \|A^{-1}\| \cdot \|y_m-y_n\|\rightarrow 0,
\end{equation*}
and $\{u_k\}_{k\in \NN}$ is a Cauchy sequence in the complete space $U$. There is a limit $u_k\rightarrow \bar{u}\in U$. By the continuity of the operator $A$, we know that $A\bar{u}=\bar{b}$, hence $\bar{b}\in \text{range}(A)$. The range of $A$ is consequently closed.

It only remains to show that $\text{range}(A)=U$. Suppose to the contrary there is a $\bar{b} \neq 0$ with $\bar{b} \in (\overline{\text{range}(A)})^\perp$. Since $\calN(A^*) = (\overline{\text{range}(A)})^\perp$, we know that 
\[
\langle A^*b,w \rangle_U = \langle b,Aw \rangle_U = 0
\]
 By the coercivity condition, we must have $0=\langle A\bar{b},\bar{b}\rangle_U\geq C_2 \|\bar{b}\|_U^2 \not = 0$. But this is a contradiction and $\text{range}(A)$ is all of $U$.
\end{proof}

Next, we discuss how error bounds are derived for Galerkin approximations $u_N$ of the solution $u$ of the operator equations above.  Let $U_N\subseteq U$ be a finite dimensional subspace of $U$. By definition, the Galerkin approximation $u_N\in U_N$ is the unique solution of the equation 
\begin{equation}
a(u_N,v_N)=\langle b,v_N\rangle_U \quad \text{ for all } v_N\in U_N. \label{eq:galerkin}
\end{equation}
The theorem below summarizes one  of the well-known bounds on the error $u-u_N$ between the Galerkin approximation $u_N\in U_N$ and the true solution $u\in U$. 
\begin{theorem}[Cea's Lemma, \cite{ciarlet2013linear}]
\label{th:galerkinerror}
Suppose that the hypotheses of the Lax-Milgram Theorem \ref{th:laxmilgram} hold. There is a unique solution $u_N\in U_N$ of the Galerkin Equation  \ref{eq:galerkin}.  The error $u-u_N$ is $a$-orthogonal to the subspace $U_N$ in the sense that 
\[
a(u-u_N,v_N)=0 \quad \text{ for all } v_N\in U_N.
\]
We also have the error bound 
\[
\|u-u_N\|_U\leq \frac{C_1}{C_2} \min_{v_N\in U_N} \|u-v_N\|_U=\frac{C_1}{C_2}\|(I-\Pi_N)u\|_U
\]
where $\Pi_N$ is the $U$-orthogonal projection of $U$ onto $U_N$. 
\end{theorem}
\begin{proof}
First, note that when $a(\cdot,\cdot)$ satisfies the boundedness and coercivity conditions, its restriction $a:U_N\times U_N \rightarrow \RR$ to $U_N$ satisfies the boundedness and coercivity conditions with the same constants relative to $U_N$.  This means that the Galerkin equations have a unique solution $u_N\in U_N$. Since Equation \ref{eq:weakeq2} holds for all $v\in U$, it holds for all $v_N\in U_N$. We can subtract Equations \ref{eq:weakeq2} and \ref{eq:galerkin} for each $v_N\in U_N$ and obtain 
\[
a(u-u_N,v_N)=0
\]
for each $v_N\in U_N$. Using the boundedness and coercivity of the bilinear form, as well as the $a-$orthogonality of the error, we have 
\begin{align*}
    C_2\|u-u_N\|_U^2 &\leq a(u-u_n,u-u_N)=a(u-u_n,u-v_N)\\ 
    &\leq C_1\|u-u_N\|_U \|u-v_N\|_U
\end{align*}
for any $v_N\in U_N$. The theorem now follows after canceling the common term on the right and left. 
\end{proof}

\bibliographystyle{spmpsci}
\bibliography{regression}

\section*{Statements and Declarations}
The authors declare that no funds, grants, or other support were received during the preparation of this manuscript. The authors have no relevant financial or non-financial interests to disclose. The datasets generated during and/or analysed during the current study are available from the corresponding author on reasonable request

\end{document}